\documentclass[reqno,12pt,a4paper]{amsart}
\usepackage{amscd,amsfonts,amssymb,amsthm,latexsym}
\usepackage{mathtools}
\usepackage{bm}
\usepackage{srcltx}

\theoremstyle{plain}
\newtheorem{theorem}{Theorem}[section]
\newtheorem*{theorem*}{Theorem}
\newtheorem{corollary}{Corollary}[section]
\newtheorem{lemma}{Lemma}[section]

\theoremstyle{definition}

\newtheorem*{definition*}{Definition}

\theoremstyle{remark}

\newtheorem*{remark*}{Remark}

\numberwithin{equation}{section}

\begin{document}
\raggedbottom %Ќужно, чтобы текст вертикально на раст€гивалс€

\title[Variants of Romanoff's theorem]{Variants of Romanoff's theorem}

\author{Artyom Radomskii}

\begin{abstract} Let $\mathcal{A}=\{a_{n}\}_{n=1}^{\infty}$ and $\mathcal{B}=\{b_{n}\}_{n=1}^{\infty}$ be two sequences of positive  integers (not necessarily distinct). Under some restrictions on $\mathcal{A}$ and $\mathcal{B}$, we obtain a lower bound for a number of integers $n$ not exceeding $x$ that can be represented as a sum $n = a_i + b_j$.
\end{abstract}

 \address{HSE University, Moscow, Russian Federation}

 \email{artyom.radomskii@mail.ru}


\maketitle

\section{Introduction}

N.~P. Romanoff \cite{Romanoff} proved that the set of positive integers representable as the sum of a prime number and a power of a given integer $a>1$ has a positive lower asymptotic density. We generalize this result in the following direction. Let $\mathcal{A}=\{a_{n}\}_{n=1}^{\infty}$ and $\mathcal{B}=\{b_{n}\}_{n=1}^{\infty}$ be two sequences of positive  integers (not necessarily distinct). Under some restrictions on $\mathcal{A}$ and $\mathcal{B}$, we obtain a lower bound for a number of integers $n$ not exceeding $x$ that can be represented as a sum $n = a_i + b_j$.

Let $\mathcal{C} =\{c_{n}\}_{n=1}^{\infty}$ be a sequence of positive integers (not necessarily distinct) and $r$ be a positive integer. We define
\[
\mathcal{C}(x)=\#\{ n\in \mathbb{N}: c_n \leq x\},\quad
\mathcal{C}(x,r)=\#\{n\in \mathbb{N}: c_{n}\leq x \text{ and } c_{n}+r\in \mathcal{C}\}.
\] We prove the following general result.

\begin{theorem}\label{T.GENERAL}
Let $\mathcal{A}=\{a_n\}_{n=1}^{\infty}$ and $\mathcal{B}=\{b_n\}_{n=1}^{\infty}$ be two sequences of positive integers. Suppose that $a_{n}< a_{n+1}$ for all $n$ and $b_{n},$ $n=1, 2, \ldots,$ are not necessarily distinct. Let
\[
\textup{ord}_{\mathcal{B}}(v)=\#\{j\in \mathbb{N}: b_j=v\}<\infty
\]for all $v\in \mathbb{N}$. Also, let
\begin{equation}\label{T1:Basic.1}
\mathcal{A}(x)\asymp \frac{x}{\eta(x)},\ \qquad\mathcal{A}(x/2)\gg \mathcal{A}(x)
\end{equation}and
\begin{equation}\label{BASIC.A.r}
\mathcal{A}(x,r) \ll \frac{r}{\varphi(r)} \frac{x}{\eta^{2}(x)}
\end{equation}for any positive integer $r$. Let
\begin{equation}
\mathcal{B}(x/2)\gg \mathcal{B}(x)\label{T1:Basic.2}
\end{equation}and
\begin{equation}\label{T1:Basic.3}
\sum_{\substack{j\in \mathbb{N}:\\ b_{j}\leq x}} \sum_{p\leq (\log x)^{\alpha}}\frac{\lambda_{\mathcal{B}} (x;j,p)\log p}{p} \ll \mathcal{B}(x)^{2}
\end{equation}for some $\alpha\in (0,1)$, where
\[\lambda_{\mathcal{B}}(x;j, p) = \#\{k\in \mathbb{N}: b_k \leq x\text{ and } b_k\equiv b_j\textup{ (mod $p$)}\}.
\]Set
\[
r(n)=\#\{(i,j)\in \mathbb{N}^{2}: a_i+b_j = n\}\quad\text{and}\quad \rho_{\mathcal{B}}(x)=\max_{1\leq v\leq x} \textup{ord}_{\mathcal{B}}(v).
\]Then there exist positive constants $c_1$ and $c_2$ depending only on $\alpha$ and the constants implied by the symbols $\ll$, $\gg$, and $\asymp$ in \eqref{T1:Basic.1} -- \eqref{T1:Basic.3} such that
\[
\#\Big\{n\leq x: r(n)\geq c_1 \frac{\mathcal{B}(x)}{\eta(x)}\Big\}\geq c_2 x \frac{\mathcal{B}(x)}{\mathcal{B}(x)+ \rho_{\mathcal{B}} (x)\eta(x)}.
\]
\end{theorem}

We will give some examples of sequences $\mathcal{A}$ that satisfy the hypotheses of Theorem \ref{T.GENERAL}.

\emph{Example} 1. Let $\mathcal{A}=\mathbb{P}$ be the set of prime numbers. By the prime number theorem
\[
\mathcal{A}(x)=\#\{p\leq x\} \sim \frac{x}{\log x}.
\]It is well-known (see, for example, \cite[Corollary 2.4.1]{Halberstam.Richert}) that
\[
\mathcal{A}(x, r)=\#\{p\leq x: p+r\in \mathbb{P}\}\ll \prod_{p|r} \Big(1-\frac{1}{p}\Big)^{-1} \frac{x}{(\log x)^{2}}= \frac{r}{\varphi (r)} \frac{x}{(\log x)^{2}},
\]where the constant implied by the $\ll$-symbol is absolute. Hence, the sequence $\mathcal{A}=\mathbb{P}$ satisfies the hypotheses \eqref{T1:Basic.1} and \eqref{BASIC.A.r} with $\eta(x)=\log x$.

\emph{Example} 2. Let $\mathcal{A} = \mathcal{S}'$ denote the set of odd integers $n$ that can be represented as the sum $n=u^{2} + v^{2}$ with $(u,v)=1$. To ease notation, we let $\langle P_1 \rangle$ denote the set of integers composed only of primes congruent to $1$ (mod $4$). It is well-known that $\mathcal{S}' = \langle P_1\rangle$ and (see, for example, \cite[Theorem 14.2]{Friedlander.Iwaniec})
\[
\mathcal{S}'(x)\sim  \frac{cx}{\sqrt{\log x}},
\]where
\[
c=\frac{1}{2\sqrt{2}}\prod_{p \equiv 3\textup{\,(mod $4$)}} \Big(1 - \frac{1}{p^{2}}\Big)^{-1/2}.
\] Also (see, for example, \cite[Corollary 2.3.4]{Halberstam.Richert})
\[
\mathcal{S}' (x, r) \ll \prod_{\substack{p|r\\ p \equiv 3 \textup{\,(mod $4$)}}} \Big(1 - \frac{1}{p}\Big)^{-1} \frac{x}{\log x} \leq \frac{r}{\varphi (r)} \frac{x}{\log x},
\]where the constant implied by the $\ll$-symbol is absolute. We see that the sequence $\mathcal{A}=\mathcal{S}'$ satisfies the hypotheses \eqref{T1:Basic.1} and \eqref{BASIC.A.r} with $\eta(x)=(\log x)^{1/2}$.

\begin{theorem}\label{Th.A.a.f.m}
Suppose that a sequence $\mathcal{A}=\{a_{n}\}_{n=1}^{\infty}$ satisfies the hypotheses of Theorem \ref{T.GENERAL}. Let $a\geq 2$ be an integer, and $f(n)=\gamma_{d}n^{d}+\ldots + \gamma_0$ be a polynomial with integer coefficients such that $\gamma_{d}>0$, $(\gamma_d,\ldots, \gamma_0)=1$, and $f: \mathbb{N} \to \mathbb{N}$. Put
\[
r(n)=\#\big\{ (k,m)\in \mathbb{N}^{2}: a_{k} + a^{f(m)}= n\big\}.
\]Then there exist positive constants $x_{0}$, $c_1$, and $c_2$ depending only on $a$, $f$, and the constants implied by the symbols $\ll$, $\gg$, $\asymp$ in \eqref{T1:Basic.1} and \eqref{BASIC.A.r} such that
\[
\#\Big\{n\leq x: r(n)\geq c_1 \frac{(\log x)^{1/d}}{\eta(x)}\Big\}\geq c_2 x \frac{(\log x)^{1/d}}{(\log x)^{1/d}+ \eta(x)}
\]for any $x \geq x_{0}$.
\end{theorem} Let $\mathcal{S}$ denote the set of integers representable as a sum of two squares. From Theorem \ref{Th.A.a.f.m} we obtain the following result.
 \begin{corollary}\label{COROLLARY.S.a.f.m}
 Let $a\geq 2$ be an integer, and $f(n)=\gamma_{d}n^{d}+\ldots + \gamma_0$ be a polynomial with integer coefficients such that $\gamma_{d}>0$, $(\gamma_d,\ldots, \gamma_0)=1$, and $f: \mathbb{N} \to \mathbb{N}$. Then there exist positive constants $x_0$, $c_1$, and $c_2$ depending only on $a$ and $f$ such that the following holds.

1) If $d\geq 2$, then
\begin{align}
\frac{c_1 x}{(\log x)^{1/2 - 1/d}}\leq \#\big\{n\leq x&: \text{ there exist $s\in \mathcal{S}$ and $m\in \mathbb{N}$} \notag\\
&\text{\ \ \ \,such that $s+a^{f(m)}=n$}\big\} \leq \frac{c_2 x}{(\log x)^{1/2 - 1/d}}\label{SQUARE.BASIC.EST}
\end{align} for any $x \geq x_0$.

2) If $d = 1$, then
\begin{equation}\label{LOW.d.1.a.f.m}
\#\{n\leq x: r(n)\geq c_{1} (\log x)^{1/2}\}\geq c_{1} x
\end{equation}for any $x \geq x_0$, where
\[
r(n):=\#\big\{(s,m)\in \mathcal{S}\times \mathbb{N}: s+ a^{f(m)}=n\big\}.
\]
\end{corollary}

The following result is true.
 \begin{theorem}\label{T.f.P}
 Suppose that a sequence $\mathcal{A}=\{a_{n}\}_{n=1}^{\infty}$ satisfies the hypotheses of Theorem \ref{T.GENERAL}. Let $f(n)=\gamma_{d}n^{d}+\ldots + \gamma_0$ be a polynomial with integer coefficients such that $\gamma_{d}>0$ and $f: \mathbb{N} \to \mathbb{N}$. Set
\[
r(n)=\#\{ (k,p)\in \mathbb{N}\times\mathbb{P}: a_{k}+ f(p)=n\}.
\]
 Then there exist positive constants $x_0$, $c_1$, and $c_2$ depending only on $f$ and the constants implied by the symbols $\ll$, $\gg$, $\asymp$ in \eqref{T1:Basic.1} and \eqref{BASIC.A.r} such that
\[
\#\Big\{n\leq x: r(n)\geq c_{1}\frac{x^{1/d}}{\eta (x) \log x}\Big\} \geq c_{2} x \frac{x^{1/d}}{x^{1/d} + \eta(x)\log x}
\]for any $x \geq x_0$.
\end{theorem}

 As applications of Theorem \ref{T.f.P}, we obtain
 \begin{corollary}\label{COROLLARY.f.P}
Let $f(n)=\gamma_{d}n^{d}+\ldots + \gamma_0$ be a polynomial with integer coefficients such that $\gamma_{d}>0$ and $f: \mathbb{N} \to \mathbb{N}$. Set
\[
r_{1}(n)=\#\{ (p,q)\in \mathbb{P}^{2}: p+ f(q)=n\}
\]and
\[
r_{2}(n)=\#\{ (s,p)\in \mathcal{S}\times \mathbb{P}: s+ f(p)=n\}.
\]
 Then there exist positive constants $x_0$, $c_1$, and $c_2$ depending only on $f$ such that
\begin{equation}\label{COROL.P}
\#\Big\{n\leq x: r_{1}(n)\geq c_{1}\frac{x^{1/d}}{(\log x)^{2}}\Big\} \geq c_{2} x
\end{equation}and
\begin{equation}\label{COROL.S}
\#\Big\{n\leq x: r_{2}(n)\geq c_{1}\frac{x^{1/d}}{(\log x)^{3/2}}\Big\} \geq c_{2} x
\end{equation}for any $x \geq x_0$.
\end{corollary}

 We obtain the following result.

 \begin{theorem}\label{T.a.f.p}
Suppose that a sequence $\mathcal{A}=\{a_{n}\}_{n=1}^{\infty}$ satisfies the hypotheses of Theorem \ref{T.GENERAL}.
Let $a\geq 2$ be an integer, and $f(n)=\gamma_{d}n^{d}+\ldots + \gamma_0$ be a polynomial with integer coefficients such that $\gamma_{d}>0$, $(\gamma_d,\ldots, \gamma_0)=1$, and $f: \mathbb{N} \to \mathbb{N}$. Set
\[
r(n)=\#\{(k,p)\in \mathbb{N}\times \mathbb{P}: a_{k} + a^{f(p)}=n\}.
\]Then there exist positive constants $x_0$, $c_1$, and $c_2$ depending only on $a$, $f$ and the constants implied by the symbols $\ll$, $\gg$, $\asymp$ in \eqref{T1:Basic.1} and \eqref{BASIC.A.r}  such that
\[
\#\Big\{n\leq x: r(n)\geq c_{1} \frac{ (\log x)^{1/d}}{\eta (x)\log\log x}\Big\}\geq c_{2} x
\frac{(\log x)^{1/d}}{(\log x)^{1/d}+ \eta(x)\log\log x}
\]for any $x \geq x_0$.
\end{theorem}

 From Theorem \ref{T.a.f.p} we deduce
 \begin{corollary}\label{CORROLARY.a.f.p}
Let $a\geq 2$ be an integer, and $f(n)=\gamma_{d}n^{d}+\ldots + \gamma_0$ be a polynomial with integer coefficients such that $\gamma_{d}>0$, $(\gamma_d,\ldots, \gamma_0)=1$, and $f: \mathbb{N} \to \mathbb{N}$. Set
\[
N(x) = \#\big\{n\leq x: \text{ there exist $s\in\mathcal{S}$ and $p\in \mathcal{\mathbb{P}}$  such that $s+a^{f(p)}=n$}\big\}
\]and
\[
r(n)=\#\{(s,p)\in \mathcal{S}\times \mathbb{P}: s+ a^{f(p)}=n\}.
\]Then there exist positive constants $x_0$, $c_1$, and $c_2$ depending only on $a$ and $f$ such that the following holds.

1) If $d\geq 2$, then
\begin{equation}\label{N.a.f.p.c1.c2}
\frac{c_{1}x}{(\log x)^{1/2-1/d}\log\log x} \leq N(x) \leq \frac{c_{2}x}{(\log x)^{1/2-1/d}\log\log x}
\end{equation}for any $x \geq x_0$.

2) If $d = 1$, then
\begin{equation}\label{r.a.f.p}
\#\Big\{n\leq x: r(n)\geq \frac{c_{1} (\log x)^{1/2}}{\log\log x}\Big\}\geq c_{1} x
\end{equation}for any $x \geq x_0$.
\end{corollary}

We prove

\begin{theorem}\label{T.f.S}
 Suppose that a sequence $\mathcal{A}=\{a_{n}\}_{n=1}^{\infty}$ satisfies the hypotheses of Theorem \ref{T.GENERAL}. Let $f(n)=\gamma_{d}n^{d}+\ldots + \gamma_0$ be a polynomial with integer coefficients such that $\gamma_{d}>0$ and $f: \mathbb{N} \to \mathbb{N}$. Set
\[
r(n)=\#\{ (k,s)\in \mathbb{N}\times \mathcal{S}: a_{k}+ f(s)=n\}.
\]
 Then there exist positive constants $x_0$, $c_1$, and $c_2$ depending only on $f$ and the constants implied by the symbols $\ll$, $\gg$, $\asymp$ in \eqref{T1:Basic.1} and \eqref{BASIC.A.r} such that
\begin{equation}\label{r.EST.s}
\#\Big\{n\leq x: r(n)\geq c_{1}\frac{x^{1/d}}{\eta (x) (\log x)^{1/2}}\Big\} \geq c_{2} x \frac{x^{1/d}}{x^{1/d} + \eta(x)(\log x)^{1/2}}
\end{equation}for any $x \geq x_0$.
\end{theorem}

 From Theorem \ref{T.f.S} we obtain
 \begin{corollary}\label{COROLLARY.f.S}
Let $f(n)=\gamma_{d}n^{d}+\ldots + \gamma_0$ be a polynomial with integer coefficients such that $\gamma_{d}>0$ and $f: \mathbb{N} \to \mathbb{N}$. Put
\[
r_{1}(n)=\#\{ (p,s)\in \mathbb{P}\times\mathcal{S}: p+ f(s)=n\}
\]and
\[
r_{2}(n)=\#\{ (s,l)\in \mathcal{S}^{2}: s+ f(l)=n\}.
\]
 Then there exist positive constants $c_{1}$, $c_{2}$, and $x_0$ depending only on $f$ such that
\begin{equation}\label{basic.s.1}
\#\Big\{n\leq x: r_{1}(n)\geq c_{1}\frac{x^{1/d}}{(\log x)^{3/2}}\Big\} \geq c_{2} x
\end{equation}and
\begin{equation}\label{basic.s.2}
\#\Big\{n\leq x: r_{2}(n)\geq c_{1}\frac{x^{1/d}}{\log x}\Big\} \geq c_{2} x
\end{equation}for any $x \geq x_0$.
\end{corollary}

Also, we prove the following result.

\begin{theorem}\label{T.a.f.s}
Suppose that a sequence $\mathcal{A}=\{a_{n}\}_{n=1}^{\infty}$ satisfies the hypotheses of Theorem \ref{T.GENERAL}.
Let $a\geq 2$ be an integer, and $f(n)=\gamma_{d}n^{d}+\ldots + \gamma_0$ be a polynomial with integer coefficients such that $\gamma_{d}>0$, $(\gamma_d,\ldots, \gamma_0)=1$, and $f: \mathbb{N} \to \mathbb{N}$. Set
\[
r(n)=\#\{(k,s)\in \mathbb{N}\times \mathcal{S}: a_{k} + a^{f(s)}=n\}.
\]Then there exist positive constants $x_0$, $c_1$, and $c_2$ depending only on $a$, $f$ and the constants implied by the symbols $\ll$, $\gg$, $\asymp$ in \eqref{T1:Basic.1} and \eqref{BASIC.A.r}  such that
\begin{equation}\label{a.f.s.STATE}
\#\Big\{n\leq x: r(n)\geq c_{1} \frac{ (\log x)^{1/d}}{\eta (x)(\log\log x)^{1/2}}\Big\}\geq c_{2} x
\frac{(\log x)^{1/d}}{(\log x)^{1/d}+ \eta(x)(\log\log x)^{1/2}}
\end{equation}for any $x \geq x_0$.
\end{theorem}

 As applications, we obtain
 \begin{corollary}\label{CORROLARY.a.f.s}
Let $a\geq 2$ be an integer, and $f(n)=\gamma_{d}n^{d}+\ldots + \gamma_0$ be a polynomial with integer coefficients such that $\gamma_{d}>0$, $(\gamma_d,\ldots, \gamma_0)=1$, and $f: \mathbb{N} \to \mathbb{N}$. Set
\[
N(x) = \#\big\{n\leq x: \text{ there exist $s, l\in\mathcal{S}$  such that $s+a^{f(l)}=n$}\big\}
\]and
\[
r(n)=\#\{(s,l)\in \mathcal{S}^{2}: s+ a^{f(l)}=n\}.
\]Then there exist positive constants $x_0$, $c_1$, and $c_2$ depending only on $a$ and $f$ such that the following holds.

1) If $d\geq 2$, then
\begin{equation}\label{N.a.f.s.c1.c2}
\frac{c_{1}x}{(\log x)^{1/2-1/d}(\log\log x)^{1/2}} \leq N(x) \leq \frac{c_{2}x}{(\log x)^{1/2-1/d}(\log\log x)^{1/2}}
\end{equation}for any $x \geq x_0$.

2) If $d = 1$, then
\begin{equation}\label{r.a.f.s}
\#\Big\{n\leq x: r(n)\geq c_{1}\Big(\frac{\log x}{\log\log x}\Big)^{1/2}\Big\}\geq c_{1} x
\end{equation}for any $x \geq x_0$.
\end{corollary}

We recall some facts on elliptic curves (for more details, see, for example, \cite[Chapter XXV]{Hardy.Wright}). An \emph{elliptic curve} is given by an equation of the form
\[
E: y^2 = x^3 + Ax+B,
\]with the one further requirement that the discriminant
\[
\Delta= 4A^3+ 27 B^2
\]should not vanish. For convenience, we shall generally assume that the coefficients $A$ and $B$ are integers. One of the properties that make an elliptic curve $E$ such a fascinating object is the existence of a composition law that allows us to 'add' points to one another. We adjoin an idealized point $\mathcal{O}$ to the plane. This point $\mathcal{O}$ is called the point at infinity. The special rules relating to the point $\mathcal{O}$ are
\[
P+(-P)=\mathcal{O}\quad \text{and}\quad P+\mathcal{O}=\mathcal{O}+P=P
\]for all points $P$ on $E$. For a prime $p$, by $\mathbb{F}_{p}$ we denote the field of congruence classes of residues modulo $p$. For $p\geq 3$, we define
\[
E(\mathbb{F}_{p})=\{(x,y)\in {\mathbb{F}}_{p}^2: y^2\equiv x^3+Ax+B\ \text{(mod $p$)}\}\cup\{\mathcal{O}\}.
\]Repeated addition and negation allows us to 'multiply' points of $E$ by an arbitrary integer $m$. This function from $E$ to itself is called the multiplication-by-$m$ map,
\[
\phi_{m}: E\to E,\qquad \phi_{m}(P)=mP=\text{\textup{sign}}(m)(P+\dots+P)
\](the sum contains $|m|$ terms). By convention, we also define $\phi_{0}(P)=\mathcal{O}$. The multiplication-by-$m$ map is defined by rational functions. Maps $E\to E$ defined by rational functions and sending $\mathcal{O}$ to $\mathcal{O}$  are called endomorphisms of $E$. For most elliptic curves (over the field of complex numbers $\mathbb{C}$), the only endomorphisms are the multiplication-by-$m$ maps. Curves that admit additional endomorphisms are said to have complex multiplication.

\begin{theorem}\label{T.ELLIPTIC.POLYNOMIAL}
Suppose that a sequence $\mathcal{A}=\{a_{n}\}_{n=1}^{\infty}$ satisfies the hypotheses of Theorem \ref{T.GENERAL}.
Let $E$ be an elliptic curve given by the equation $y^{2}= x^{3} +Ax + B$, where $A$ and $B$ are integers satisfying $\Delta = 4A^{3} + 27 B^{2}\neq 0$. Suppose that $E$ does not have complex multiplication. Let $f(n)=\gamma_{d} n^{d}+\ldots + \gamma_{0}$ be a polynomial with integer coefficients such that $\gamma_{d}>0$ and $f: \mathbb{N}\to \mathbb{N}$. For any positive integer $n$, we set
\[
r(n)=\#\{ (k, p)\in \mathbb{N}\times\mathbb{P}: p\geq 3\text{ and } a_{k}+ f(\#E(\mathbb{F}_{p}))= n \}.
\]Then there exist positive constants $x_0$, $c_1$, and $c_2$ depending only on $E$, $f$ and the constants implied by the symbols $\ll$, $\gg$, $\asymp$ in \eqref{T1:Basic.1} and \eqref{BASIC.A.r} such that
\[
\#\Big\{n\leq x: r(n) \geq c_1 \frac{x^{1/d}}{\eta(x)\log x} \Big\} \geq c_2 x \frac{x^{1/(2d)}}{x^{1/(2d)}+ \eta(x)\log x}
\]for any $x \geq x_0$.
\end{theorem}

From Theorem \ref{T.ELLIPTIC.POLYNOMIAL} we obtain
\begin{corollary}\label{COROLLARY.ELLIPTIC.POLYNOMIAL}
Let $E$ be an elliptic curve given by the equation $y^{2}= x^{3} +Ax + B$, where $A$ and $B$ are integers satisfying $\Delta = 4A^{3} + 27 B^{2}\neq 0$. Suppose that $E$ does not have complex multiplication. Let $f(n)=\gamma_{d} n^{d}+\ldots + \gamma_{0}$ be a polynomial with integer coefficients such that $\gamma_{d}>0$ and $f: \mathbb{N}\to \mathbb{N}$. For any positive integer $n$, we set
\[
r_{1}(n)=\#\{ (p, q)\in \mathbb{P}^{2}: q\geq 3\text{ and } p+ f(\#E(\mathbb{F}_{q}))= n \}
\]and
\[
r_{2}(n)=\#\{ (s, p)\in \mathcal{S}\times \mathbb{P}: p\geq 3\text{ and } s+ f(\#E(\mathbb{F}_{p}))= n \}.
\]Then there exist positive constants $x_0$, $c_1$, and $c_2$ depending only on $E$ and $f$ such that
\begin{equation}\label{REFER.ELL.1}
\#\Big\{n\leq x: r_{1}(n) \geq c_1 \frac{x^{1/d}}{(\log x)^{2}} \Big\} \geq c_2 x
\end{equation}and
\begin{equation}\label{REFER.ELL.2}
\#\Big\{n\leq x: r_{2}(n) \geq c_1 \frac{x^{1/d}}{(\log x)^{3/2}} \Big\} \geq c_2 x
\end{equation}for any $x \geq x_0$.
\end{corollary} Corollary \ref{COROLLARY.ELLIPTIC.POLYNOMIAL} extends Theorem 1.8 in \cite{Radomskii.Izv} which showed only  \eqref{REFER.ELL.1} with $f(n)=n$.

It is clear that $1\leq \varphi(n)\leq n$ for any positive integer $n$. Therefore, if $a_1,\ldots, a_{N}$ are positive integers (not necessarily distinct), $s\in \mathbb{N}$, then
 \[
\sum_{n=1}^{N} \biggl(\frac{a_{n}}{\varphi(a_{n})}\biggr)^{s}\geq N.
\] Upper bounds for such sums are of interest. We prove the following result.

\begin{theorem}\label{T.EULER.ELLIPTIC}
Let $E$ be an elliptic curve given by the equation $y^{2}= x^{3} +Ax + B$, where $A$ and $B$ are integers satisfying $\Delta = 4A^{3} + 27 B^{2}\neq 0$. Suppose that $E$ does not have complex multiplication. Let $f(n)=\gamma_{d} n^{d}+\ldots + \gamma_{0}$ be a polynomial with integer coefficients such that $\gamma_{d}>0$ and $f: \mathbb{N}\to \mathbb{N}$. Then there exists a constant $C=C(E,f)>0$ depending only on $E$ and $f$ such that for any $s\in \mathbb{N}$ we have
 \begin{equation}\label{EULER.BASIC.1}
\sum_{3 \leq p \leq x} \bigg(\frac{f(\#E(\mathbb{F}_{p}))}{\varphi(f(\#E(\mathbb{F}_{p})))}\bigg)^{s}
\leq \textup{exp}(s\log s +Cs)\pi(x)
\end{equation}for all $x\geq 3$.
\end{theorem}

\section{Notation}

We use $X\ll Y$, $Y\gg X$, or $X=O(Y)$ to denote the estimate $|X|\leq C Y$ for some constant $C>0$, and write $X\asymp Y$ for $X \ll Y \ll X$. The notation $X\sim Y$  means $\lim_{x\to \infty} X/Y = 1$.

 Let $(a,b)$ denote the greatest common divisor of integers $a$ and $b$. We write $\mathbb{Z}$ for the set of integers, $\mathbb{N}$ for the set of positive integers, $\mathbb{P}$ for the set of prime numbers, $\mathcal{S}$ for the set of numbers representable as the sum of two squares, and $\mathcal{S}'$ for the set of odd integers $n$ that can be represented as the sum $n=u^{2} + v^{2}$ with $(u,v)=1$.

 We reserve the letters $p$ and $q$ for primes. In particular, the sum $\sum_{p\leq K}$ should be interpreted as being over all prime numbers not exceeding $K$. By $\mathbb{F}_{p}$ we denote the field of congruence classes of residues modulo $p$. We write $\nu(n)$ for the number of distinct prime divisors of $n$ and $\varphi (n)$ for Euler's totient function (the order of the multiplicative group of reduced residue classes modulo $n$).

For a prime $p$ and an integer $a$ with $(a,p)=1$, let $h_{a}(p)$ denote the order of $a$ modulo $p$, which is to say that $h_{a}(p)$ is the least positive integer $h$ such that $a^{h}\equiv 1$ (mod $p$). By Fermat's theorem, $a^{p-1}\equiv 1$ (mod $p$), and hence $h_{a}(p)$ exists and $1\leq h_{a}(p) \leq p-1$. It is well-known that $h_{a}(p)| (p-1)$.

If $x$ is a real number, then $[x]$ denotes its integral part and $\{x\}= x - [x]$ its fractional part. By $\# A$ we denote the number of elements of a finite set $A$. We denote the number of primes not exceeding $x$ by $\pi (x)$ and the number of primes not exceeding $x$ that are congruent to $l$ modulo $k$ by $\pi (x; k, l)$.

\section{Proof of Theorem \ref{T.GENERAL}}

\begin{lemma}\label{Lemma.Euler}
Let $0<\alpha <1$. Then there exists a constant $C(\alpha)>0$ depending only on $\alpha$ such that the following holds. Let $a_1,\ldots, a_{N}$ be positive integers \textup{(}not necessarily distinct\textup{)} with $a_{n}\leq M$ for all $1\leq n \leq N$. Set
\[
\omega(d)=\#\{1\leq n\leq N: a_{n}\equiv 0\ \textup{\text{(mod $d$)}}\}
\]for any positive integer $d$. Let $s$ be a positive integer. Then
\[
\sum_{n=1}^{N}\left(\frac{a_{n}}{\varphi(a_{n})}\right)^{s}\leq (C(\alpha))^{s}
\biggl(N+\sum_{p\leq (\log M)^{\alpha}} \frac{\omega(p)(\log p)^{s}}{p}\biggr).
\]
\begin{proof} This is \cite[Theorem 1.1]{Radomskii.Izv}.
\end{proof}
\end{lemma}

We first obtain an upper bound for $\sum_{n\leq x} r(n)^{2}$. We observe that
\[
\sum_{n \leq x} r(n)^{2} = \sum_{\substack{(i_1, i_2, j_1, j_2)\in \mathbb{N}^{4}:\\
a_{i_{1}}+ b_{j_{1}} \leq x\\ a_{i_1} + b_{j_1} = a_{i_2} + b_{j_2} }}1= \sum_{\substack{\ldots\\ i_{1}=i_{2}}}1+
\sum_{\substack{\ldots\\ i_{1}<i_{2}}}1 + \sum_{\substack{\ldots\\ i_{1}>i_{2}}}1 = S_{1}+ S_{2}+ S_{3}.
\] We have
\begin{align}
S_{1} = \sum_{\substack{i\in \mathbb{ N}:\\ a_{i}\leq x}} \sum_{\substack{j\in \mathbb{N}:\\ b_{j}\leq x- a_{i}}}
\sum_{\substack{k\in \mathbb{N}:\\ b_{k}= b_{j}}}1 &=
\sum_{\substack{i\in \mathbb{ N}:\\ a_{i}\leq x}} \sum_{\substack{j\in \mathbb{N}:\\ b_{j}\leq x- a_{i}}} \textup{ord}_{\mathcal{B}}(b_j)\notag\\
&\leq \rho_{\mathcal{B}}(x)\mathcal{B}(x)\mathcal{A}(x)\ll \frac{x}{\eta(x)} \rho_{\mathcal{B}}(x)\mathcal{B}(x).\label{S1.FINAL}
\end{align}

It is easy to see that $S_{2} = S_{3}$. Now we estimate $S_{2}$. We see that
\begin{equation}\label{S2.EST.UP}
S_2 \leq \sum_{\substack{j, k\in \mathbb{N}:\\ b_{j}< b_{k}\leq x}} \mathcal{A}(x, b_{k} - b_{j}) \ll \frac{x}{\eta^{2}(x)} \sum_{\substack{j, k\in \mathbb{N}:\\ b_{j}< b_{k} \leq x}} \frac{b_{k} - b_{j}}{\varphi (b_{k}- b_{j})},
\end{equation}by \eqref{BASIC.A.r}.

Fix $j\in \mathbb{N}$ such that $b_{j} < x$. Applying Lemma \ref{Lemma.Euler} with $s=1$, $\alpha$ given by \eqref{T1:Basic.3} and $M=x$, we obtain
\begin{align*}
\sum_{\substack{k\in \mathbb{N}:\\ b_{j}< b_{k} \leq x}} \frac{b_{k} - b_{j}}{\varphi (b_{k}- b_{j})}
&\leq C(\alpha) \Big(\sum_{\substack{k\in \mathbb{N}:\\ b_{j}< b_{k} \leq x}}1 + \sum_{p \leq (\log x)^{\alpha}}\frac{\lambda_{\mathcal{B}} (x;j, p)\log p}{p}\Big)\\
&\ll \mathcal{B}(x) + \sum_{p \leq (\log x)^{\alpha}}\frac{\lambda_{\mathcal{B}} (x;j, p)\log p}{p}.
\end{align*} From \eqref{S2.EST.UP} and \eqref{T1:Basic.3} we get
\begin{equation}\label{S2.EST.UP.FINAL}
S_{2} \ll \frac{x}{\eta(x)^{2}}\sum_{\substack{j\in \mathbb{N}:\\ b_{j} < x}} \Big(\mathcal{B}(x) + \sum_{p \leq (\log x)^{\alpha}}\frac{\lambda_{\mathcal{B}} (x;j, p)\log p}{p}\Big) \ll \frac{x}{\eta(x)^{2}}\mathcal{B}(x)^{2}.
\end{equation}We find from \eqref{S1.FINAL} and \eqref{S2.EST.UP.FINAL} that
\begin{equation}\label{r.SECOND.MOMENT}
\sum_{n \leq x} r(n)^{2} \ll \frac{x}{\eta(x)^{2}} \mathcal{B}(x)(\rho_{\mathcal{B}}(x) \eta(x)+ \mathcal{B}(x)).
\end{equation}

Now we obtain a lower bound for $\sum_{n\leq x} r(n)$. By \eqref{T1:Basic.1} and \eqref{T1:Basic.2} we have
\[
\sum_{n\leq x} r(n)\geq \mathcal{A}(x/2) \mathcal{B}(x/2) \geq c \frac{x}{\eta(x)} \mathcal{B}(x),
\]where $c$ is a positive constant depending only on the constants implied by the symbols $\asymp$ and $\gg$ in \eqref{T1:Basic.1} and \eqref{T1:Basic.2}. Since
\[
\sum_{\substack{n\leq x:\\ r(n)< (c \mathcal{B}(x))/(2 \eta(x))}} r(n)< \frac{c}{2}\frac{\mathcal{B}(x)}{\eta(x)} \sum_{\substack{n\leq x:\\ r(n)< (c \mathcal{B}(x))/(2 \eta(x))}} 1 \leq \frac{c}{2} \frac{x}{\eta(x)} \mathcal{B}(x),
\]we obtain
\[
\sum_{\substack{n\leq x:\\ r(n)\geq (c \mathcal{B}(x))/(2 \eta(x))}} r(n) \geq \frac{c}{2} \frac{x}{\eta(x)} \mathcal{B}(x).
\]

By Cauchy's inequality and \eqref{r.SECOND.MOMENT}, we have
\begin{align*}
\frac{c}{2} \frac{x}{\eta(x)} \mathcal{B}(x) &\leq \sum_{\substack{n\leq x:\\ r(n)\geq (c \mathcal{B}(x))/(2 \eta(x))}} r(n)
\leq \Big(\sum_{n \leq x} r(n)^{2}\Big)^{1/2} \Big(\sum_{\substack{n\leq x:\\ r(n)\geq (c \mathcal{B}(x))/(2 \eta(x))}} 1\Big)^{1/2}\\
&\ll \frac{(x \mathcal{B}(x)(\rho_{\mathcal{B}}(x)\eta(x)+ \mathcal{B}(x)))^{1/2}}{\eta(x)}\Big(\sum_{\substack{n\leq x:\\ r(n)\geq (c \mathcal{B}(x))/(2 \eta(x))}} 1\Big)^{1/2}.
\end{align*} We obtain
\[
\#\Big\{n\leq x: r(n) \geq \frac{c}{2} \frac{\mathcal{B}(x)}{\eta(x)}\Big\}\gg x \frac{\mathcal{B}(x)}{\mathcal{B}(x)+ \rho_{\mathcal{B}}(x)\eta(x)}.
\]Theorem \ref{T.GENERAL} is proved.

\section{Proof of Theorem \ref{Th.A.a.f.m} and Corollary \ref{COROLLARY.S.a.f.m}}

\begin{lemma}\label{L1}
Let $d$ and $m$ be positive integers. Let
\[
f(x)=\sum_{i=0}^{d} b_{i} x^{i},
\]where $b_{0},\ldots, b_d$ are integers with $(b_{0},\ldots, b_d, m)=1$. Let $\rho (f,m)$ denote the number of solutions of the congruence $f(x)\equiv 0$ \textup{(mod $m$)}. Then
\[
\rho (f, m)\leq c d m^{1-1/d},
\]where $c>0$ is an absolute constant.
\end{lemma}
\begin{proof}
This is \cite[Theorem 2]{Konyagin}.
\end{proof}

\begin{lemma}\label{LEMMA.h.a.p}
Let $a\geq 2$ be an integer and let $\varepsilon>0$. Then
\[
\sum_{\substack{p:\\ (p,a)=1}}\frac{\log p}{p(h_{a}(p))^{\varepsilon}}\ll 1,
\]where the constant implied by the $\ll$-symbol depends only on $a$ and $\varepsilon$.
\end{lemma}
\begin{proof} Suppose that $N\geq 100$. We have
\[
\sum_{p|N} \frac{\log p}{p} \leq \sum_{p\leq \log N} \frac{\log p}{p} + \sum_{\substack{p|N\\ p> \log N}} \frac{\log p}{p}= S_{1}+S_{2}.
\]By Mertens' Theorem,
\[
S_{1} = \log\log N + O(1).
\] Since the function $h(x) = \log x / x$ is decreasing on the interval $(e, +\infty)$, we obtain
\[
S_{2} \leq \frac{\log\log N}{\log N}\,\nu (N) \leq c_{0}.
\] We have
\begin{equation}\label{MERTENS.LOG}
\sum_{p|N} \frac{\log p}{p} \leq \log\log N +c_{1} \leq c \log\log N,
\end{equation}where $c= c_{1}+1>0$ is an absolute constant.

We set
\[
z_{n} = \sum_{\substack{p:\\ (p,a)=1\\ h_{a}(p)=n}} \frac{\log p}{ p },\qquad G(t)=\sum_{n\leq t} z_{n},
\qquad P(t)=\prod_{n \leq t} (a^{n}-1).
\]Assume that $t$ is sufficiently large (in terms of $a$ and $\varepsilon$). If $h_{a}(p) = n \leq t$, then $p|P(t)$. Since $P(t) \leq a^{t^{2}}$, from \eqref{MERTENS.LOG} we obtain
\[
G(t)\leq \sum_{p|P(t)} \frac{\log p}{p} \ll \log \log P(t)\ll \log t.
\]

Applying summation by parts, we get
\[
\sum_{\substack{p:\\ (p,a)=1}} \frac{\log p}{p(h_{a}(p))^{\varepsilon}} = \sum_{n=1}^{\infty}
\sum_{\substack{p:\\ (p,a)=1\\h_{a}(p)=n}} \frac{\log p}{p(h_{a}(p))^{\varepsilon}}=
\sum_{n=1}^{\infty}\frac{z_{n}}{n^\varepsilon} = \varepsilon\int_{1}^{+\infty}\frac{G(t)}{t^{1+\varepsilon}}\, dt
\ll 1.
\] Lemma \ref{LEMMA.h.a.p} is proved.
\end{proof}

\begin{lemma}\label{L2}
 Let $a\geq 2$ be an integer, and $f(n)=\gamma_{d}n^{d}+\ldots + \gamma_0$ be a polynomial with integer coefficients such that $\gamma_{d}>0$, $(\gamma_d,\ldots, \gamma_0)=1$, and $f: \mathbb{N} \to \mathbb{N}$. Set $\mathcal{B}=\{b_{n}\}_{n=1}^{\infty}$, where $b_n = a^{f(n)}$. Then for sufficiently large $x$ (in terms of $a$ and $f$) we have
 \begin{equation}\label{B.x.EST.POLYNOMIAL}
 \mathcal{B}(x)\asymp (\log x)^{1/d},\qquad 1\leq \rho_{\mathcal{B}}(x) \leq d,
 \end{equation}and
 \begin{equation}\label{LAMBDA.POLYNOMIAL.EST}
\sum_{\substack{j\in \mathbb{N}:\\ b_j \leq x}} \sum_{p\leq (\log x)^{1/(2d)}} \frac{\lambda_{\mathcal{B}} (x;j, p)\log p}{p}
\ll  \mathcal{B}(x)^{2},
\end{equation} where the constants implied by the symbols $\asymp$ and $\ll$ depend only on $a$ and $f$.
\end{lemma}

\begin{proof}
 Since the equation $f(x)=c$ has at most $d$ roots, we have $\textup{ord}_{\mathcal{B}}(v) \leq d$ for any positive integer $v$. Hence
 \[
 1\leq \rho_{\mathcal{B}}(x) \leq d,
 \]if $x$ is sufficiently large in terms of $a$ and $f$.

  There exist positive constants $c_1$ and $c_2$ depending only on $f$ such that
\begin{equation}\label{POLYNOMIAL.c.1.c.2}
c_1 n^{d} \leq f(n) \leq c_2 n^{d}
\end{equation}for any $n\in \mathbb{N}$. We have
\begin{align*}
\mathcal{B}(x) &=\#\big\{n\in \mathbb{N}: a^{f(n)}\leq x\big\}=
\#\Big\{n\in \mathbb{N}: f(n)\leq \frac{\log x}{\log a}\Big\}\\
&\leq \#\Big\{n\in \mathbb{N}: c_{1} n^{d}\leq \frac{\log x}{\log a}\Big\}
\ll (\log x)^{1/d}.
\end{align*} Similarly,
\[
\mathcal{B}(x) \geq \#\Big\{n\in \mathbb{N}: c_{2} n^{d}\leq \frac{\log x}{\log a}\Big\}\gg (\log x)^{1/d}.
\] Thus, \eqref{B.x.EST.POLYNOMIAL} is proved.

Fix $j\in \mathbb{N}$ such that $b_j \leq x$ and $p\leq (\log x)^{1/ (2 d)}$. Since
\[
\lambda_{\mathcal{B}} (x;j, p)=\#\{k\in \mathbb{N}: b_{k}\leq x\text{ and $b_{k} \equiv b_{j}$ (mod $p$)}\}\leq \mathcal{B}(x),
\] we get
\begin{equation}\label{Ex.3.Sum.1}
\sum_{\substack{p\leq (\log x)^{1/(2d)}\\ p|a}} \frac{\lambda_{\mathcal{B}} (x;j, p)\log p}{p} \leq \mathcal{B}(x)
\sum_{p|a} \frac{\log p}{p}=c(a) \mathcal{B}(x).
\end{equation}

 Suppose that $(p,a)=1$. Let $\alpha_{1},\ldots, \alpha_{l}$ be all solutions of the congruence $f(x) \equiv f(j)$ (mod $h_{a}(p)$). By Lemma \ref{L1}, we have
 \begin{equation}\label{l.est.good}
 l\leq c d (h_{a}(p))^{1-1/d}.
 \end{equation}  Applying \eqref{POLYNOMIAL.c.1.c.2}, we obtain
\begin{align}
\lambda_{\mathcal{B}} (x;j, p) &= \#\{ k\in \mathbb{N}: a^{f(k)}\leq x\text{ and $a^{f(k)}\equiv a^{f(j)}$ (mod $p$)}\}\notag\\
&\leq \#\{ 1\leq k \leq [(\log x / (c_1 \log a))^{1/d}]: f(k)\equiv f(j)\text{ (mod $h_{a}(p)$)}\}\notag\\
&=\sum_{i=1}^{l} \#\{ 1\leq k \leq [(\log x / (c_1 \log a))^{1/d}]: k\equiv \alpha_{i}\text{ (mod $h_{a}(p)$)}\}.\label{lambda.est.good}
\end{align}

 Since
\[
h_{a}(p)\leq p-1 < (\log x)^{1/(2d)}\leq [(\log x / (c_1 \log a))^{1/d}],
\]for any $1\leq i \leq l$ we have (see \eqref{B.x.EST.POLYNOMIAL})
\[
\#\{ 1\leq k \leq [(\log x / (c_1 \log a))^{1/d}]: k\equiv \alpha_{i}\text{ (mod $h_{a}(p)$)}\}\ll \frac{(\log x)^{1/d}}{h_{a}(p)} \ll \frac{\mathcal{B}(x)}{h_{a}(p)}.
\]From \eqref{l.est.good} and \eqref{lambda.est.good} we obtain
\[
\lambda_{\mathcal{B}}(x;j, p) \ll \frac{\mathcal{B}(x)}{(h_{a}(p))^{1/d}},
\]where the constant implied by the $\ll$-symbol depends on $a$ and $f$ only.

Using Lemma \ref{LEMMA.h.a.p}, we obtain
\begin{align}
\sum_{\substack{p\leq (\log x)^{1/(2d)}\\ (p,a)=1}} \frac{\lambda_{\mathcal{B}} (x;j, p)\log p}{p}
&\ll \mathcal{B}(x) \sum_{\substack{p\leq (\log x)^{1/(2d)}\\ (p,a)=1}} \frac{\log p}{p (h_{a}(p))^{1/d}}\notag\\
&\leq \mathcal{B}(x) \sum_{\substack{p:\\ (p,a)=1}} \frac{\log p}{ p (h_{a}(p))^{1/d}}
\ll \mathcal{B}(x).\label{Ex.3.Sum.2}
\end{align}Gathering \eqref{Ex.3.Sum.1} and \eqref{Ex.3.Sum.2} together, we get
\[
\sum_{p\leq (\log x)^{1/(2d)}} \frac{\lambda_{\mathcal{B}} (x;j, p)\log p}{p}
\ll \mathcal{B}(x),
\] and hence
\[
\sum_{\substack{j\in \mathbb{N}:\\ b_j\leq x}} \sum_{p\leq (\log x)^{1/(2d)}} \frac{\lambda_{\mathcal{B}} (x;j, p)\log p}{p}
\ll \mathcal{B}(x) \sum_{\substack{j\in \mathbb{N}:\\ b_j \leq x}} 1= \mathcal{B}(x)^{2},
\] where the constant implied by the $\ll$-symbol depends only on $a$ and $f$. Lemma \ref{L2} is proved.
 \end{proof}
 We see from \eqref{B.x.EST.POLYNOMIAL} and \eqref{LAMBDA.POLYNOMIAL.EST}  that the sequence $\mathcal{B}$ satisfies the hypotheses \eqref{T1:Basic.2} and \eqref{T1:Basic.3} with $\alpha = 1/(2d)$. Theorem \ref{Th.A.a.f.m} follows from Theorem \ref{T.GENERAL} and Lemma \ref{L2}.

\begin{proof}[Proof of Corollary \ref{COROLLARY.S.a.f.m}.] Suppose that $d\geq 2$. Let
 \[
 N(x) = \#\{n\leq x:\text{ there exist $s\in \mathcal{S}$ and $m\in \mathbb{N}$ such that $s+a^{f(m)}=n$}\}
 \]and
 \[
 N'(x) = \#\{n\leq x:\text{ there exist $s\in \mathcal{S}'$ and $m\in \mathbb{N}$ such that $s+a^{f(m)}=n$}\}.
 \]
 Set
 \[
 \mathcal{A}=\mathcal{S}'\qquad\text{and}\qquad \mathcal{B}=\big\{a^{f(m)}: m\in \mathbb{N}\big\}.
 \]
  By Example 2, the sequence $\mathcal{A}$ satisfies the hypotheses \eqref{T1:Basic.1} and \eqref{BASIC.A.r} with $\eta(x)=(\log x)^{1/2}$.
 Applying Theorem \ref{Th.A.a.f.m}, we obtain
 \[
 N'(x) \gg \frac{x}{(\log x)^{1/2 - 1/d}},
 \]where the constant implied by the $\gg$-symbol depends only on $a$ and $f$. Since $N'(x)\leq N(x)$, the first inequality in \eqref{SQUARE.BASIC.EST} is proved. Also,
 \[
 N(x)\leq \mathcal{S}(x) \mathcal{B}(x)\ll \frac{x}{(\log x)^{1/2}} (\log x)^{1/d} = \frac{x}{(\log x)^{1/2 - 1/d}}.
 \]

 Now suppose that $d=1$, and let
 \[
 r'(n)=\#\{(s,m)\in \mathcal{S}'\times \mathbb{N}: s+a^{f(m)}= n\}.
 \]Using Theorem \ref{Th.A.a.f.m}, we obtain
 \[
 \#\{n\leq x: r'(n) \geq c_{1}(\log x)^{1/2}\}\gg x.
 \] Since $r'(n)\leq r(n)$, we obtain \eqref{LOW.d.1.a.f.m}. This completes the proof of Corollary \ref{COROLLARY.S.a.f.m}.
 \end{proof}

 \section{Proof of Theorem \ref{T.f.P} and Corollary \ref{COROLLARY.f.P}}

 \begin{lemma}[The Brun - Titchmarsh theorem]\label{BRUN.TITCHMARSCH}
 Let $0< a < 1$, $x\geq 2$, $1\leq k \leq x^{a}$, and $(l,k)=1$. Then there is a constant $C(a)>0$ depending only on $a$ such that
 \[
 \pi (x;k,l) < C(a) \frac{x}{\varphi (k)\log x}.
 \]
 \end{lemma}
 \begin{proof}
 This Lemma follows immediately from \cite[Theorem 3.8]{Halberstam.Richert}.
 \end{proof}

 \begin{lemma}\label{LEMMA.f.PRIME}
  Let $f(n)=\gamma_{d} n^{d}+\ldots + \gamma_{0}$ be a polynomial with integer coefficients such that $\gamma_{d}>0$ and $f: \mathbb{N}\to \mathbb{N}$. Set $\mathcal{B}=\{f(p): p\in \mathbb{P}\}$. Then for sufficiently large $x$ (in terms of $f$) we have
 \begin{equation}\label{B.x.POLYN.LEMMA}
 \mathcal{B}(x)\asymp \frac{x^{1/d}}{\log x},\qquad 1\leq \rho_{\mathcal{B}}(x) \leq d,
 \end{equation}and
 \[
\sum_{\substack{j\in \mathbb{P}:\\ f(j) \leq x}} \sum_{p\leq (\log x)^{1/2}} \frac{\lambda_{\mathcal{B}} (x;j, p)\log p}{p}
\ll  \mathcal{B}(x)^{2},
\] where the constants implied by the symbols $\asymp$ and $\ll$ depend only on the polynomial $f$.
 \end{lemma}
 \begin{proof}
 Since the equation $f(x)=v$ has at most $d$ roots, we see that $\textup{ord}_{\mathcal{B}}(v)\leq d$ for any positive integer $v$. Therefore
 \[
 1\leq \rho_{\mathcal{B}}(x) \leq d,
 \]if $x$ is sufficiently large in terms of $f$. Applying \eqref{POLYNOMIAL.c.1.c.2}, we have
 \begin{equation}\label{B.PRIMES.LOW}
 \mathcal{B}(x)=\#\{p: f(p)\leq x\}\leq \#\{p: c_{1}p^{d}\leq x\}\ll \frac{x^{1/d}}{\log x},
 \end{equation}and
 \[
 \mathcal{B}(x)\geq \#\{p: c_{2}p^{d}\leq x\}\gg \frac{x^{1/d}}{\log x},
 \]where the constants implied by the symbols $\ll$ and $\gg$ depend only on $f$. Thus, \eqref{B.x.POLYN.LEMMA} is proved.

 Fix $j\in \mathbb{P}$ such that $f(j)\leq x$ and $p\leq (\log x)^{1/2}$. Suppose that $p> \gamma_{d}$, where $\gamma_{d}$ is the leading coefficient of the polynomial $f$. Let $\alpha_{1}, \ldots, \alpha_{l}$ be all solutions of the congruence $f(x) \equiv f(j)$ (mod $p$). By Lagrange's Theorem, $l\leq d$. Set
 \[
 N=\Big[\Big(\frac{x}{c_{1}}\Big)^{1/d}\Big],
 \]where $c_{1}$ is the constant from \eqref{POLYNOMIAL.c.1.c.2}. By \eqref{B.PRIMES.LOW}, we have
 \begin{align}
 \lambda_{\mathcal{B}}(x;j, p)&=\#\{q: f(q)\leq x\text{ and } f(q)\equiv f(j)\text{ (mod $p$)}\}\notag\\
 &= \sum_{i=1}^{l}\#\{q: f(q)\leq x\text{ and } q\equiv \alpha_{i}\text{ (mod $p$)}\}\notag\\
 &\leq \sum_{i=1}^{l}\#\Big\{q\leq \Big(\frac{x}{c_{1}}\Big)^{1/d}: q\equiv \alpha_{i}\text{ (mod $p$)}\Big\}=\sum_{i=1}^{l} \pi(N;p,\alpha_{i}).\label{LAMBD.POLY.EST}
 \end{align}

 Fix $1 \leq i \leq l$. We have
 \[
 p\leq (\log x)^{1/2} \leq N^{1/2}.
 \] If $(\alpha_{i}, p) = 1$, then by Lemma \ref{BRUN.TITCHMARSCH} with $a=1/2$ we have
 \[
 \pi(N;p,\alpha_{i})
 \leq C(f) \frac{x^{1/d}}{p \log x},
 \]where $C(f)>0$ is a constant depending only on $f$. If $(\alpha_{i}, p)>1$, then
 \[
 \pi(N;p,\alpha_{i})
 \leq 1 \leq C(f) \frac{x^{1/d}}{(\log x)^{3/2}} \leq C(f) \frac{x^{1/d}}{p \log x}.
 \]Hence,
 \[
 \pi(N;p,\alpha_{i})
 \leq C(f) \frac{x^{1/d}}{p \log x}
 \]in both cases $(\alpha_{i}, p)=1$ and $(\alpha_{i}, p)>1$. By \eqref{LAMBD.POLY.EST} and \eqref{B.x.POLYN.LEMMA}, we get
 \[
 \lambda_{\mathcal{B}}(x;j,p)\ll \frac{\mathcal{B}(x)}{p}.
 \]

 We obtain
 \begin{equation}\label{LAMBDA.EST.POLY.I}
 \sum_{\gamma_{d}< p \leq (\log x)^{1/2}} \frac{\lambda_{\mathcal{B}}(x;j, p)\log p}{p}
 \ll \mathcal{B}(x)\sum_{\gamma_{d}< p \leq (\log x)^{1/2}} \frac{\log p}{p^{2}}< \mathcal{B}(x)\sum_{p}\frac{\log p}{p^{2}}\ll \mathcal{B}(x).
 \end{equation}Since
 \[
 \lambda_{\mathcal{B}}(x;j,p)\leq \mathcal{B}(x),
 \]we also have
 \begin{equation}\label{LAMBDA.EST.POLY.II}
 \sum_{p\leq \gamma_{d}} \frac{\lambda_{\mathcal{B}}(x;j, p)\log p}{p}
 \leq \mathcal{B}(x)\sum_{p\leq \gamma_{d}} \frac{\log p}{p}\ll \mathcal{B}(x),
 \end{equation}where the constant implied by the symbol $\ll$ depends only on $f$. Gathering \eqref{LAMBDA.EST.POLY.I} and \eqref{LAMBDA.EST.POLY.II} together, we obtain
 \[
 \sum_{p \leq (\log x)^{1/2}} \frac{\lambda_{\mathcal{B}}(x;j, p)\log p}{p}
 \ll \mathcal{B}(x),
 \]and hence
 \[
  \sum_{\substack{j\in \mathbb{P}:\\ f(j)\leq x}}\sum_{p \leq (\log x)^{1/2}} \frac{\lambda_{\mathcal{B}}(x;j, p)\log p}{p} \ll \mathcal{B}(x)\sum_{\substack{j\in \mathbb{P}:\\ f(j)\leq x}} 1=
  \mathcal{B}(x)^{2}.
 \]Lemma \ref{LEMMA.f.PRIME} is proved.
 \end{proof}
  Theorem \ref{T.f.P} follows from Theorem \ref{T.GENERAL} and Lemma \ref{LEMMA.f.PRIME}.

\begin{proof}[Proof of Corollary \ref{COROLLARY.f.P}.] We set
 \[
 \mathcal{A}_{1}=\mathbb{P},\qquad \mathcal{A}_{2}= \mathcal{S}',\qquad \mathcal{B}=\{f(p): p\in \mathbb{P}\}.
 \]By Examples 1 and 2, the sequences $\mathcal{A}_{1}$ and $\mathcal{A}_{2}$ satisfy the hypotheses of Theorem \ref{T.f.P}.  Applying this theorem, we obtain \eqref{COROL.P} and
 \begin{equation}\label{BASIC.T.f.p.III.PRIME}
 \#\Big\{n \leq x: r'_{2}(n) \geq c_{1}\frac{x^{1/d}}{(\log x)^{3/2}}\Big\} \gg x,
 \end{equation}where
 \[
 r'_{2}(n)=\#\{ (s,p)\in \mathcal{S}'\times \mathbb{P}: s+ f(p)=n\}.
 \]Since $r'_{2}(n)\leq r_{2}(n)$, from \eqref{BASIC.T.f.p.III.PRIME} we obtain \eqref{COROL.S}. Corollary \ref{COROLLARY.f.P} is proved.
 \end{proof}

\section{Proof of Theorem \ref{T.a.f.p} and Corollary \ref{CORROLARY.a.f.p}}
\begin{lemma}\label{LEMMA.a.f.p}
Let $a\geq 2$ be an integer, and $f(n)=\gamma_{d}n^{d}+\ldots + \gamma_0$ be a polynomial with integer coefficients such that $\gamma_{d}>0$, $(\gamma_d,\ldots, \gamma_0)=1$, and $f: \mathbb{N} \to \mathbb{N}$. Set $\mathcal{B}=\{b_{p}: p\in \mathbb{P}\}$, where $b_{p}=a^{f(p)}$. Then for sufficiently large $x$ (in terms of $a$ and $f$) we have
 \begin{equation}\label{BASIC.a.f.p.ORDER.rho}
 \mathcal{B}(x)\asymp \frac{(\log x)^{1/d}}{\log\log x},\qquad 1\leq \rho_{\mathcal{B}}(x) \leq d,
 \end{equation}and
 \[
\sum_{\substack{j\in \mathbb{P}:\\ b_{j} \leq x}} \sum_{p\leq (\log x)^{1/(3d)}} \frac{\lambda_{\mathcal{B}} (x;j, p)\log p}{p}
\ll  \mathcal{B}(x)^{2},
\] where the constants implied by the symbols $\asymp$ and $\ll$ depend only on $a$ and $f$.
\end{lemma}
\begin{proof} Since the equation $f(x)=c$ has at most $d$ roots, we obtain $\textup{ord}_{\mathcal{B}}(v) \leq d$ for any positive integer $v$. Hence
\[
1\leq \rho_{\mathcal{B}}(x) \leq d,
\]if $x$ is sufficiently large in terms of $a$ and $f$ (it is enough for $x$ to be greater than $a^{f(2)}$). Applying \eqref{POLYNOMIAL.c.1.c.2}, we have
\begin{align*}
\mathcal{B}(x) &=\#\big\{p\in \mathbb{P}: a^{f(p)}\leq x\big\}=
\#\Big\{p\in \mathbb{P}: f(p)\leq \frac{\log x}{\log a}\Big\}\\
&\leq \#\Big\{p\in \mathbb{P}: c_{1} p^{d}\leq \frac{\log x}{\log a}\Big\}=
\#\Big\{p\in \mathbb{P}: p\leq \Big(\frac{\log x}{c_{1}\log a}\Big)^{1/d}\Big\}
\ll \frac{(\log x)^{1/d}}{\log\log x}.
\end{align*} Similarly,
\[
\mathcal{B}(x) \geq \#\Big\{p\in \mathbb{P}: c_{2} p^{d}\leq \frac{\log x}{\log a}\Big\}\gg \frac{(\log x)^{1/d}}{\log\log x}.
\] Thus, \eqref{BASIC.a.f.p.ORDER.rho} is proved.

Fix $j\in \mathbb{P}$ such that $b_j \leq x$ and $p\leq (\log x)^{1/ (3 d)}$. Since $\lambda_{\mathcal{B}} (x;j, p)\leq \mathcal{B}(x)$, we get
\begin{equation}\label{FIRST.EST.a.f.p}
\sum_{\substack{p\leq (\log x)^{1/(3d)}\\ p|a}} \frac{\lambda_{\mathcal{B}} (x;j, p)\log p}{p} \leq \mathcal{B}(x)
\sum_{p|a} \frac{\log p}{p}=c(a) \mathcal{B}(x).
\end{equation}

 Suppose that $(p,a)=1$. We put
  \[
  N= \Big[\Big(\frac{\log x}{c_{1}\log a}\Big)^{1/d}\Big].
  \]Let $\alpha_{1},\ldots, \alpha_{l}$ be all solutions of the congruence $f(x) \equiv f(j)$ (mod $h_{a}(p)$). By Lemma \ref{L1}, we have
 \begin{equation}\label{l.EST.a.f.p}
 l\leq c d (h_{a}(p))^{1-1/d}.
 \end{equation}  Applying \eqref{POLYNOMIAL.c.1.c.2}, we obtain
\begin{align}
\lambda_{\mathcal{B}} (x;j, p) &= \#\{ q\in \mathbb{P}: a^{f(q)}\leq x\text{ and $a^{f(q)}\equiv a^{f(j)}$ (mod $p$)}\}\notag\\
&\leq \#\{ q \leq [(\log x / (c_1 \log a))^{1/d}]: f(q)\equiv f(j)\text{ (mod $h_{a}(p)$)}\}\notag\\
&=\sum_{i=1}^{l} \#\{ q \leq N: q\equiv \alpha_{i}\text{ (mod $h_{a}(p)$)}\}=\sum_{i=1}^{l} \pi (N; h_{a}(p),\alpha_{i}).\label{LAMBDA.SUMMA.a.f.p}
\end{align}

 Fix $1\leq i \leq l$. We have
 \[
h_{a}(p)\leq p-1 < (\log x)^{1/(3d)}\leq N^{1/2}.
\] Suppose that $(\alpha_{i}, h_{a}(p))=1$. Applying Lemma \ref{BRUN.TITCHMARSCH} with $a=1/2$, we obtain
 \[
\pi(N;h_{a}(p),\alpha_{i})\leq C(a,f) \frac{(\log x)^{1/d}}{\varphi(h_{a}(p))\log\log x},
\]where $C(a,f)>0$ is a constant depending only on $a$ and $f$. If $(\alpha_{i}, h_{a}(p))> 1$, then
\begin{align*}
\pi(N;h_{a}(p),\alpha_{i})\leq 1&\leq C(a,f) \frac{(\log x)^{1/d}}{(\log x)^{1/(3d)}\log\log x}\\
&\leq C(a,f) \frac{(\log x)^{1/d}}{\varphi(h_{a}(p))\log\log x}.
\end{align*}Thus, in both cases $(\alpha_{i}, h_{a}(p))=1$ and $(\alpha_{i}, h_{a}(p))>1$ we have
\begin{equation}
\pi(N;h_{a}(p),\alpha_{i})\leq C(a,f) \frac{(\log x)^{1/d}}{\varphi(h_{a}(p))\log\log x}
\ll \frac{\mathcal{B}(x)}{\varphi (h_{a}(p))}.\label{LAMBDA.EST.a.f.p}
\end{equation}

From \eqref{l.EST.a.f.p}, \eqref{LAMBDA.SUMMA.a.f.p}, and \eqref{LAMBDA.EST.a.f.p} we obtain
\[
\lambda_{\mathcal{B}}(x;j, p) \ll \frac{\mathcal{B}(x)(h_{a}(p))^{1 - 1/d}}{\varphi (h_{a}(p))},
\]where the constant implied by the $\ll$-symbol depends on $a$ and $f$ only. Since
\[
\varphi(n) \geq c\,\frac{n}{\log\log (n+2)}
\]for any positive integer $n$, there exists a positive constant $c(d)$ depending only on $d$ such that
\[
\frac{n}{\varphi (n)} \leq c(d)n^{1/(2d)}
\]for any $n\in \mathbb{N}$. Therefore
\[
\lambda_{\mathcal{B}}(x;j, p) \ll \frac{\mathcal{B}(x)}{(h_{a}(p))^{1/(2d)}}.
\]

Using Lemma \ref{LEMMA.h.a.p}, we find
\begin{align}
\sum_{\substack{p\leq (\log x)^{1/(3d)}\\ (p,a)=1}} \frac{\lambda_{\mathcal{B}} (x;j, p)\log p}{p}
&\ll \mathcal{B}(x) \sum_{\substack{p\leq (\log x)^{1/(3d)}\\ (p,a)=1}} \frac{\log p}{p(h_{a}(p))^{1/(2d)}}\notag\\
&\leq \mathcal{B}(x) \sum_{\substack{p:\\ (p,a)=1}} \frac{\log p}{p(h_{a}(p))^{1/(2d)}}
\ll \mathcal{B}(x).\label{SECOND.EST.a.f.p}
\end{align} Gathering \eqref{FIRST.EST.a.f.p} and \eqref{SECOND.EST.a.f.p} together, we get
\[
\sum_{p\leq (\log x)^{1/(3d)}} \frac{\lambda_{\mathcal{B}} (x;j, p)\log p}{p}
\ll \mathcal{B}(x),
\] and hence
\[
\sum_{\substack{j\in \mathbb{P}:\\ b_j\leq x}} \sum_{p\leq (\log x)^{1/(3d)}} \frac{\lambda_{\mathcal{B}} (x;j, p)\log p}{p}
\ll \mathcal{B}(x) \sum_{\substack{j\in \mathbb{P}:\\ b_j \leq x}} 1= \mathcal{B}(x)^{2},
\] where the constant implied by the $\ll$-symbol depends only on $a$ and $f$. Lemma \ref{LEMMA.a.f.p} is proved.
\end{proof} Theorem \ref{T.a.f.p} follows from Theorem \ref{T.GENERAL} and Lemma \ref{LEMMA.a.f.p}.

\begin{proof}[Proof of Corollary \ref{CORROLARY.a.f.p}.] We set
 \[
\mathcal{A}= \mathcal{S}',\qquad \mathcal{B}=\Big\{a^{f(p)}: p\in \mathbb{P}\Big\}.
 \]By Example 2, the sequences $\mathcal{A}$ satisfies the hypotheses of Theorem \ref{T.a.f.p}. Applying this theorem, in the case $d \geq 2$ we have
 \[
 N'(x)\geq \frac{c_{1} x}{(\log x)^{1/2 -1/d}\log\log x},
 \]where
 \[
 N'(x)=\#\big\{n\leq x:\text{ there exist $s\in \mathcal{S}'$ and $p\in \mathbb{P}$ such that $s+ a^{f(p)}=n$}\big\}.
 \] Since $N'(x)\leq N(x)$, we obtain the first inequality in \eqref{N.a.f.p.c1.c2}. Also,
 \[
 N(x) \leq \mathcal{S}(x) \mathcal{B}(x)\ll \frac{x}{(\log x)^{1/2}} \frac{(\log x)^{1/d}}{\log\log x}
 =\frac{x}{(\log x)^{1/2 - 1/d}\log\log x},
 \]and \eqref{N.a.f.p.c1.c2} is proved.

 If $d=1$, then by Theorem \ref{T.a.f.p} we have
\begin{equation}\label{r.strich.a.f.p}
 \#\Big\{n \leq x: r'(n) \geq \frac{c_{1}(\log x)^{1/2}}{\log\log x}\Big\} \gg x,
 \end{equation}where
 \[
 r'(n)=\#\big\{ (s,p)\in \mathcal{S}'\times \mathbb{P}: s+ a^{f(p)}=n\big\}.
 \]Since $r'(n)\leq r(n)$, from \eqref{r.strich.a.f.p} we obtain \eqref{r.a.f.p}. Corollary \ref{CORROLARY.a.f.p} is proved.
 \end{proof}

\section{Proof of Theorem \ref{T.f.S} and Corollary \ref{COROLLARY.f.S}}

\begin{lemma}\label{LEMMA.IWANIEC}
Let $k$ be a positive integer, and let $l$ be an integer satisfying the conditions $(k,l)=1$ and $l\equiv 1$ \textup{(mod $(4,k)$)}. Then uniformly in $k$ one has
\begin{align*}
\#\{&n\leq x: n\equiv l\textup{ (mod $k$) and }n\in\mathcal{S}\} \\
&= \frac{(4,k)}{(2,k) k}\prod_{\substack{p|k\\ p \equiv 3\textup{\,(mod $4$)}}} (1+1/p)\frac{Bx}{\sqrt{\log x}}\Bigg(1+ O\bigg(\frac{\log (2k)}{\log x}\bigg)^{1/5}\Bigg),
\end{align*}where
\[
B=\frac{1}{\sqrt{2}}\prod_{p \equiv 3\textup{\,(mod $4$)}} \Big(1 - \frac{1}{p^{2}}\Big)^{-1/2}
\]and the implied constant is absolute.
\end{lemma}
\begin{proof}
This is Lemma 2.1 in \cite{Balog.Wooley} (which, in turn, follows from Iwaniec \cite[Corollary 1]{Iwaniec} and Rieger \cite[Satz 1]{Rieger}).
\end{proof}

\begin{lemma}\label{LEMMA.s.PROGRESSION}
Let $0< a< 1,$ $x \geq 2$, $1 \leq k \leq x^{a}$, and $l\in \mathbb{Z}$. Then there is a constant $C(a)>0$ depending only on $a$ such that
\[
\#\{n\leq x: n\equiv l\textup{ (mod $k$) and }n\in\mathcal{S}'\}\leq \frac{C(a)\sqrt{\log\log (k+2)}}{k} \frac{x}{\sqrt{\log x}}.
\]
\end{lemma}
\begin{proof}It suffices to assume that $x\geq x_{0}(a)$, where $x_{0}(a)>0$ is a large constant depending only on $a$. We have
\begin{align*}
\log \prod_{\substack{p|k\\ p \equiv 3\textup{\,(mod $4$)}}} (1+1/p)&\leq \sum_{\substack{p|k\\ p \equiv 3\textup{\,(mod $4$)}}} \frac{1}{p}\leq \sum_{\substack{p\leq \log(k+100)\\ p \equiv 3\textup{\,(mod $4$)}}} \frac{1}{p}+ \sum_{\substack{p|k\\ p> \log (k+100)}}\frac{1}{p}\\
&\leq \frac{1}{2}\log \log\log (k+100) + O(1) + \frac{\nu(k)}{\log (k+100)}\\
&\leq \frac{1}{2}\log \log\log (k+100) + c_0,
\end{align*}where $c_{0}>0$ is an absolute constant. We obtain
\[
\prod_{\substack{p|k\\ p \equiv 3\textup{\,(mod $4$)}}} (1+1/p)\leq c \sqrt{\log \log (k+2)}.
\]Therefore, if $(l,k)=1$ and $l\equiv 1$ (mod $(4,k)$), from
Lemma \ref{LEMMA.IWANIEC} we obtain
\begin{equation}\label{S.PROGRESS}
\#\{n\leq x: n\equiv l\textup{ (mod $k$) and }n\in\mathcal{S}'\}\leq \frac{C(a)\sqrt{\log\log (k+2)}}{k} \frac{x}{\sqrt{\log x}},
\end{equation} where $C(a)>0$ is a constant depending only on $a$.

Let us denote
\[
\Lambda = \{n\leq x: n\equiv l\textup{ (mod $k$) and }n\in\mathcal{S}'\}\qquad\text{and} \qquad T=\#\Lambda .
\] If $n\in \Lambda$, then $n\equiv 1$ (mod $4$) (since $n\in \langle P_{1}\rangle$). We see that if $l\not\equiv 1$ (mod $(4,k)$), then $T=0$ and \eqref{S.PROGRESS} follows. Hence, we can assume that $l\equiv 1$ (mod $(4,k)$). Suppose that $\delta = (l,k)>1$. If $n\in \Lambda$, then $\delta | n$. Therefore, if $\delta\notin \langle P_{1}\rangle$, then $T=0$ and \eqref{S.PROGRESS} follows. Hence, we can assume that $\delta \in \langle P_{1}\rangle$.

 Let us denote
\[
l_{1} = \frac{l}{\delta},\qquad k_{1}=\frac{k}{\delta},\qquad b= (4,k),\qquad \text{and}\qquad b_{1}= (4,k_{1}).
\]First of all,  we note that $(l_{1}, k_{1})=1$ and $b=b_{1}$, since $\delta \in \langle P_{1}\rangle$. Let us show that
\begin{equation}\label{l.1.b.1}
l_{1} \equiv 1\textup{ (mod $b_{1}$)}.
\end{equation} If $b=1$, then \eqref{l.1.b.1} holds. Suppose that $b=2$. If $l_{1}$ is even, then $l$ is even, but this contradicts the fact that $l\equiv 1$ (mod $b$). Hence, \eqref{l.1.b.1} holds in the case $b=2$. Finally, suppose that $b=4$. We note that $\delta \equiv 1$ (mod $4$), since $\delta \in \langle P_{1}\rangle$. If $l_{1} \equiv \alpha$ (mod $4$) for some $\alpha\in\{2, 3, 4\}$, then $l\equiv \alpha$ (mod $4$), but this contradicts the fact that $l\equiv 1$ (mod $b$). Thus, \eqref{l.1.b.1} is proved.

If $n\in \Lambda$, then $n=\delta n'$, where $n'\leq x/\delta$, $n'\in \mathcal{S}'$, and $n' \equiv l_{1}$ (mod $k_{1}$). Since $\delta\leq k \leq x^{a}$, we have
\[
\frac{x}{\delta} \geq x^{1-a}\geq 2,
\]if $x_{0}(a)$ is sufficiently large. Also,
\[
\delta^{a} k_{1} \leq \delta k_{1}= k \leq x^{a}
\]and hence $k_{1} \leq (x/\delta)^{a}$. Since $(l_{1}, k_{1})=1$ and $l_{1} \equiv 1$ (mod $(4,k_{1})$), from \eqref{S.PROGRESS} we obtain
\begin{align*}
T\leq \#\Big\{n'\leq \frac{x}{\delta}: n'\equiv l_{1}\textup{ (mod $k_{1}$) and }n'\in\mathcal{S}'\Big\}
&\leq \frac{C(a)\sqrt{\log\log (k_{1}+2)}}{k_{1}} \frac{x/\delta}{\sqrt{\log (x/\delta)}}\\
&\leq \frac{C'(a)\sqrt{\log\log (k+2)}}{k} \frac{x}{\sqrt{\log x}},
\end{align*}where $C' (a)>0$ is a constant depending only on $a$. Lemma \ref{LEMMA.s.PROGRESSION} is proved.
\end{proof}

\begin{lemma}\label{LEMMA.f.s}
  Let $f(n)=\gamma_{d} n^{d}+\ldots + \gamma_{0}$ be a polynomial with integer coefficients such that $\gamma_{d}>0$ and $f: \mathbb{N}\to \mathbb{N}$. Set $\mathcal{B}=\{f(s): s\in \mathcal{S}'\}$. Then for sufficiently large $x$ (in terms of $f$) we have
 \begin{equation}\label{B.x.POLYN.LEMMA.s}
 \mathcal{B}(x)\asymp \frac{x^{1/d}}{(\log x)^{1/2}},\qquad 1\leq \rho_{\mathcal{B}}(x) \leq d,
 \end{equation}and
 \[
\sum_{\substack{j\in \mathcal{S}':\\ f(j) \leq x}} \sum_{p\leq (\log x)^{1/2}} \frac{\lambda_{\mathcal{B}} (x;j, p)\log p}{p}
\ll  \mathcal{B}(x)^{2},
\] where the constants implied by the symbols $\asymp$ and $\ll$ depend only on the polynomial $f$.
 \end{lemma}
 \begin{proof}
 Since the equation $f(x)=c$ has at most $d$ roots, we obtain that $\textup{ord}_{\mathcal{B}}(v)\leq d$ for any positive integer $v$. Therefore
 \[
 1\leq \rho_{\mathcal{B}}(x) \leq d,
 \]if $x$ is sufficiently large in terms of $f$. Applying \eqref{POLYNOMIAL.c.1.c.2}, we have
 \begin{equation}\label{B.PRIMES.LOW.s}
 \mathcal{B}(x)=\#\{s\in \mathcal{S}': f(s)\leq x\}\leq \#\{s\in \mathcal{S}': c_{1}s^{d}\leq x\}\ll \frac{x^{1/d}}{(\log x)^{1/2}},
 \end{equation}and
 \[
 \mathcal{B}(x)\geq \#\{s\in \mathcal{S}': c_{2}s^{d}\leq x\}\gg \frac{x^{1/d}}{(\log x)^{1/2}},
 \]where the constants implied by the symbols $\ll$ and $\gg$ depend only on $f$. Thus, \eqref{B.x.POLYN.LEMMA.s} is proved.

 Fix $j\in \mathcal{S}'$ such that $f(j)\leq x$, and $p\leq (\log x)^{1/2}$. Suppose that $p> \gamma_{d}+3$, where $\gamma_{d}$ is the leading coefficient of the polynomial $f$. Let $\alpha_{1}, \ldots, \alpha_{l}$ be all solutions of the congruence $f(x) \equiv f(j)$ (mod $p$). By Lagrange's Theorem, $l\leq d$. By \eqref{B.PRIMES.LOW.s}, we have
 \begin{align*}
 \lambda_{\mathcal{B}}(x;j, p)&=\#\{s\in\mathcal{S}': f(s)\leq x\text{ and } f(s)\equiv f(j)\text{ (mod $p$)}\}\notag\\
 &= \sum_{i=1}^{l}\#\{s\in \mathcal{S}': f(s)\leq x\text{ and } s\equiv \alpha_{i}\text{ (mod $p$)}\}\notag\\
 &\leq \sum_{i=1}^{l}\#\Big\{s\leq \Big(\frac{x}{c_{1}}\Big)^{1/d}: s\equiv \alpha_{i}\text{ (mod $p$)}\text{ and } s\in\mathcal{S}'\Big\}.
 \end{align*}

 Fix $1 \leq i \leq l$. We have
 \[
 p\leq (\log x)^{1/2}< \Big(\frac{x}{c_{1}}\Big)^{1/(2d)}.
 \] Applying Lemma \ref{LEMMA.s.PROGRESSION} with $a=1/2$, we obtain
 \begin{align*}
 \#\Big\{s\leq \Big(\frac{x}{c_{1}}\Big)^{1/d}: s\equiv \alpha_{i}\text{ (mod $p$)}\text{ and } s\in\mathcal{S}'\Big\}
 &\leq C(f) \frac{x^{1/d}(\log\log p)^{1/2}}{p (\log x)^{1/2}}\\
 &\ll \mathcal{B}(x)\frac{(\log\log p)^{1/2}}{p},
 \end{align*}where $C(f)>0$ is a constant depending only on $f$. Hence,
 \[
 \lambda_{\mathcal{B}}(x;j, p) \ll \mathcal{B}(x)\frac{(\log\log p)^{1/2}}{p}.
 \]

 We obtain
 \begin{align}
 \sum_{\gamma_{d}+3< p \leq (\log x)^{1/2}} \frac{\lambda_{\mathcal{B}}(x;j, p)\log p}{p}
 &\ll \mathcal{B}(x)\sum_{\gamma_{d}+3< p \leq (\log x)^{1/2}} \frac{\log p(\log\log p)^{1/2}}{p^{2}}\notag\\
 &< \mathcal{B}(x)\sum_{p\geq 3}\frac{\log p(\log\log p)^{1/2}}{p^{2}}\ll \mathcal{B}(x).\label{LAMBDA.EST.POLY.I.s}
 \end{align}Since
 \[
 \lambda_{\mathcal{B}}(x;j,p)\leq \mathcal{B}(x),
 \]we also have
 \begin{equation}\label{LAMBDA.EST.POLY.II.s}
 \sum_{p\leq \gamma_{d}+3} \frac{\lambda_{\mathcal{B}}(x;j, p)\log p}{p}
 \leq \mathcal{B}(x)\sum_{p\leq \gamma_{d}+3} \frac{\log p}{p}\ll \mathcal{B}(x),
 \end{equation}where the constant implied by the $\ll$-symbol depends only on $f$. Gathering \eqref{LAMBDA.EST.POLY.I.s} and \eqref{LAMBDA.EST.POLY.II.s} together, we obtain
 \[
 \sum_{p \leq (\log x)^{1/2}} \frac{\lambda_{\mathcal{B}}(x;j, p)\log p}{p}
 \ll \mathcal{B}(x),
 \]and hence
 \[
  \sum_{\substack{j\in \mathcal{S}':\\ f(j)\leq x}}\sum_{p \leq (\log x)^{1/2}} \frac{\lambda_{\mathcal{B}}(x;j, p)\log p}{p} \ll \mathcal{B}(x)\sum_{\substack{j\in \mathcal{S}':\\ f(j)\leq x}} 1=
  \mathcal{B}(x)^{2}.
 \]Lemma \ref{LEMMA.f.s} is proved.
 \end{proof}We complete the proof of Theorem \ref{T.f.S}. Let
 \[
 r'(n)=\#\{(k,s)\in \mathbb{N}\times \mathcal{S}': a_{k}+ f(s)=n\}.
 \]It follows from Theorem \ref{T.GENERAL} and Lemma \ref{LEMMA.f.s} that there are positive constants $x_{0}$, $c_{1}$, and $c_{2}$ depending only on $f$ and the constants implied by the symbols $\ll$, $\gg$, $\asymp$ in \eqref{T1:Basic.1} and \eqref{BASIC.A.r} such that
 \begin{equation}\label{r.EST.subset}
 \#\Big\{n\leq x: r'(n)\geq c_{1}\frac{x^{1/d}}{\eta(x)(\log x)^{1/2}}\Big\}\geq c_{2}x \frac{x^{1/d}}{x^{1/d}+\eta(x)(\log x)^{1/2}}
 \end{equation}for any $x\geq x_{0}$. Since $r'(n)\leq r(n)$, from \eqref{r.EST.subset} we obtain \eqref{r.EST.s}. Theorem \ref{T.f.S} is proved.

\begin{proof}[Proof of Corollary \ref{COROLLARY.f.S}.] We set
 \[
 \mathcal{A}_{1}=\mathbb{P},\qquad \mathcal{A}_{2}= \mathcal{S}',\qquad \mathcal{B}=\{f(s): s\in \mathcal{S}\}.
 \]By Examples 1 and 2, the sequences $\mathcal{A}_{1}$ and $\mathcal{A}_{2}$ satisfy the hypotheses of Theorem \ref{T.f.S}. It follows from this theorem that there exist positive constants $c_{1}$, $c_{2}$, and $x_{0}$ depending only on $f$ such that \eqref{basic.s.1} holds and
 \begin{equation}\label{BASIC.T.f.s.III}
 \#\Big\{n \leq x: r'_{2}(n) \geq  c_{1}\frac{ x^{1/d}}{\log x}\Big\}\geq c_{2} x
 \end{equation}for any $x\geq x_{0}$, where
 \[
 r'_{2}(n)=\#\{ (s,l)\in \mathcal{S}'\times \mathcal{S}: s+ f(l)=n\}.
 \]Since $r'_{2}(n)\leq r_{2}(n)$, from \eqref{BASIC.T.f.s.III} we obtain \eqref{basic.s.2}. Corollary \ref{COROLLARY.f.S} is proved.
 \end{proof}

\section{Proof of Theorem \ref{T.a.f.s} and Corollary \ref{CORROLARY.a.f.s}}
\begin{lemma}\label{LEMMA.a.f.s}
Let $a\geq 2$ be an integer, and $f(n)=\gamma_{d}n^{d}+\ldots + \gamma_0$ be a polynomial with integer coefficients such that $\gamma_{d}>0$, $(\gamma_d,\ldots, \gamma_0)=1$, and $f: \mathbb{N} \to \mathbb{N}$. Set $\mathcal{B}=\{b_{s}: s\in \mathcal{S}'\}$, where $b_{s}=a^{f(s)}$. Then for sufficiently large $x$ (in terms of $a$ and $f$) we have
 \begin{equation}\label{BASIC.a.f.p.ORDER.rho.s}
 \mathcal{B}(x)\asymp \frac{(\log x)^{1/d}}{(\log\log x)^{1/2}},\qquad 1\leq \rho_{\mathcal{B}}(x) \leq d,
 \end{equation}and
 \[
\sum_{\substack{j\in \mathcal{S}':\\ b_{j} \leq x}} \sum_{p\leq (\log x)^{1/(3d)}} \frac{\lambda_{\mathcal{B}} (x;j, p)\log p}{p}
\ll  \mathcal{B}(x)^{2},
\] where the constants implied by the symbols $\asymp$ and $\ll$ depend only on $a$ and $f$.
\end{lemma}
\begin{proof} Since the equation $f(x)=c$ has at most $d$ roots, we obtain $\textup{ord}_{\mathcal{B}}(v) \leq d$ for any positive integer $v$. Hence
\[
1\leq \rho_{\mathcal{B}}(x) \leq d,
\]if $x$ is sufficiently large in terms of $a$ and $f$. Applying \eqref{POLYNOMIAL.c.1.c.2}, we have
\begin{align*}
\mathcal{B}(x) &=\#\big\{s\in \mathcal{S}': a^{f(s)}\leq x\big\}=
\#\Big\{s\in \mathcal{S}': f(s)\leq \frac{\log x}{\log a}\Big\}\\
&\leq \#\Big\{s\in \mathcal{S}': c_{1} s^{d}\leq \frac{\log x}{\log a}\Big\}=
\#\Big\{s\in \mathcal{S}': s\leq \Big(\frac{\log x}{c_{1}\log a}\Big)^{1/d}\Big\}
\ll \frac{(\log x)^{1/d}}{(\log\log x)^{1/2}}.
\end{align*} Similarly,
\[
\mathcal{B}(x) \geq \#\Big\{s\in \mathcal{S}': c_{2} s^{d}\leq \frac{\log x}{\log a}\Big\}\gg \frac{(\log x)^{1/d}}{(\log\log x)^{1/2}}.
\] Thus, \eqref{BASIC.a.f.p.ORDER.rho.s} is proved.

Fix $j\in \mathcal{S}'$ such that $b_j \leq x$ and $p\leq (\log x)^{1/ (3 d)}$. Since $\lambda_{\mathcal{B}} (x;j, p)\leq \mathcal{B}(x)$, we get
\begin{equation}\label{FIRST.EST.a.f.p.s}
\sum_{\substack{p\leq (\log x)^{1/(3d)}\\ p|a}} \frac{\lambda_{\mathcal{B}} (x;j, p)\log p}{p} \leq \mathcal{B}(x)
\sum_{p|a} \frac{\log p}{p}=c(a) \mathcal{B}(x).
\end{equation}

 Suppose that $(p,a)=1$. We put
  \[
  N= \Big[\Big(\frac{\log x}{c_{1}\log a}\Big)^{1/d}\Big].
  \]Let $\alpha_{1},\ldots, \alpha_{l}$ be all solutions of the congruence $f(x) \equiv f(j)$ (mod $h_{a}(p)$). By Lemma \ref{L1}, we have
 \begin{equation}\label{l.EST.a.f.p.s}
 l\leq c d (h_{a}(p))^{1-1/d}.
 \end{equation}  Applying \eqref{POLYNOMIAL.c.1.c.2}, we obtain
\begin{align}
\lambda_{\mathcal{B}} (x;j, p) &= \#\{ s\in \mathcal{S}': a^{f(s)}\leq x\text{ and $a^{f(s)}\equiv a^{f(j)}$ (mod $p$)}\}\notag\\
&\leq \#\{ s \leq [(\log x / (c_1 \log a))^{1/d}]: f(s)\equiv f(j)\textup{ (mod $h_{a}(p)$) and }s\in\mathcal{S}'\}\notag\\
&=\sum_{i=1}^{l} \#\{ s \leq N: s\equiv \alpha_{i}\textup{ (mod $h_{a}(p)$) and }s\in\mathcal{S}'\}.\label{LAMBDA.SUMMA.a.f.p.s}
\end{align}

 Fix $1\leq i \leq l$. We have
 \[
h_{a}(p)\leq p-1 < (\log x)^{1/(3d)}\leq N^{1/2}.
\]  Applying Lemma \ref{LEMMA.s.PROGRESSION} with $a=1/2$, we obtain
 \[
\#\{ s \leq N: s\equiv \alpha_{i}\textup{ (mod $h_{a}(p)$) and }s\in\mathcal{S}'\}\leq C(a,f) \frac{(\log x)^{1/d}\log\log (h_{a}(p)+2)}{h_{a}(p)(\log\log x)^{1/2}}.
\]where $C(a,f)$ is a positive constant depending only on $a$ and $f$. There is a constant $c(d)>0$ depending only on $d$ such that
\[
\log\log (n+2) \leq c(d) n^{1/(2d)}
\]for any $n\in \mathbb{N}$. Using also \eqref{l.EST.a.f.p.s}, \eqref{LAMBDA.SUMMA.a.f.p.s}, and \eqref{BASIC.a.f.p.ORDER.rho.s}, we obtain
\[
\lambda_{\mathcal{B}} (x;j, p)\ll \frac{\mathcal{B}(x)}{(h_{a}(p))^{1/(2d)}},
\]where the constant implied by the $\ll$-symbol depends on $a$ and $f$ only.

Using Lemma \ref{LEMMA.h.a.p}, we find
\begin{align}
\sum_{\substack{p\leq (\log x)^{1/(3d)}\\ (p,a)=1}} \frac{\lambda_{\mathcal{B}} (x;j, p)\log p}{p}
&\ll \mathcal{B}(x) \sum_{\substack{p\leq (\log x)^{1/(3d)}\\ (p,a)=1}} \frac{\log p}{p(h_{a}(p))^{1/(2d)}}\notag\\
&\leq \mathcal{B}(x) \sum_{\substack{p:\\ (p,a)=1}} \frac{\log p}{p(h_{a}(p))^{1/(2d)}}
\ll \mathcal{B}(x).\label{SECOND.EST.a.f.p.s}
\end{align} Gathering \eqref{FIRST.EST.a.f.p.s} and \eqref{SECOND.EST.a.f.p.s} together, we get
\[
\sum_{p\leq (\log x)^{1/(3d)}} \frac{\lambda_{\mathcal{B}} (x;j, p)\log p}{p}
\ll \mathcal{B}(x),
\] and hence
\[
\sum_{\substack{j\in \mathcal{S}':\\ b_j\leq x}} \sum_{p\leq (\log x)^{1/(3d)}} \frac{\lambda_{\mathcal{B}} (x;j, p)\log p}{p}
\ll \mathcal{B}(x) \sum_{\substack{j\in \mathcal{S}':\\ b_j \leq x}} 1= \mathcal{B}(x)^{2},
\] where the constant implied by the $\ll$-symbol depends only on $a$ and $f$. Lemma \ref{LEMMA.a.f.s} is proved.
\end{proof}

We complete the proof of Theorem \ref{T.a.f.s}. Let
 \[
 r'(n)=\#\Big\{(k,s)\in \mathbb{N}\times \mathcal{S}': a_{k}+ a^{f(s)}=n\Big\}.
 \]It follows from Theorem \ref{T.GENERAL} and Lemma \ref{LEMMA.a.f.s} that there are positive constants $x_{0}$, $c_{1}$, and $c_{2}$ depending only on $a$, $f$, and the constants implied by the symbols $\ll$, $\gg$, $\asymp$ in \eqref{T1:Basic.1} and \eqref{BASIC.A.r} such that
 \begin{equation}\label{r.EST.subset.s}
 \#\Big\{n\leq x: r'(n)\geq c_{1}\frac{(\log x)^{1/d}}{\eta(x)(\log\log x)^{1/2}}\Big\}\geq c_{2}x \frac{(\log x)^{1/d}}{(\log x)^{1/d}+\eta(x)(\log\log x)^{1/2}}
 \end{equation}for any $x\geq x_{0}$. Since $r'(n)\leq r(n)$, from \eqref{r.EST.subset.s} we obtain \eqref{a.f.s.STATE}. Theorem \ref{T.a.f.s} is proved.

\begin{proof}[Proof of Corollary \ref{CORROLARY.a.f.s}.] We set
 \[
\mathcal{A}= \mathcal{S}',\qquad \mathcal{B}=\Big\{a^{f(s)}: s\in \mathcal{S}\Big\}.
 \]By Example 2, the sequences $\mathcal{A}$ satisfies the hypotheses of Theorem \ref{T.a.f.s}. Applying this theorem, in the case $d \geq 2$ we have
 \[
 N'(x)\geq \frac{c_{1} x}{(\log x)^{1/2 -1/d}(\log\log x)^{1/2}},
 \]where
 \[
 N'(x)=\#\big\{n\leq x:\text{ there exist $s\in \mathcal{S}'$ and $l\in \mathcal{S}$ such that $s+ a^{f(l)}=n$}\big\}.
 \] Since $N'(x)\leq N(x)$, we obtain the first inequality in \eqref{N.a.f.s.c1.c2}. Also,
 \[
 N(x) \leq \mathcal{S}(x) \mathcal{B}(x)\ll \frac{x}{(\log x)^{1/2}} \frac{(\log x)^{1/d}}{(\log\log x)^{1/2}}
 =\frac{x}{(\log x)^{1/2 - 1/d}(\log\log x)^{1/2}},
 \]verifying \eqref{N.a.f.s.c1.c2}.

 If $d=1$, then by Theorem \ref{T.a.f.s} we have
\begin{equation}\label{r.strich.a.f.s}
 \#\Big\{n \leq x: r'(n) \geq c_{1}\Big(\frac{\log x}{\log\log x}\Big)^{1/2}\Big\} \gg x,
 \end{equation}where
 \[
 r'(n)=\#\big\{ (s,l)\in \mathcal{S}'\times \mathcal{S}: s+ a^{f(l)}=n\big\}.
 \]Since $r'(n)\leq r(n)$, from \eqref{r.strich.a.f.s} we obtain \eqref{r.a.f.s}. Corollary \ref{CORROLARY.a.f.s} is proved.
 \end{proof}

\section{Proof of Theorem \ref{T.ELLIPTIC.POLYNOMIAL} and Corollary \ref{COROLLARY.ELLIPTIC.POLYNOMIAL}}

\begin{lemma}\label{LEMMA.DAVID.WU}
Let $E$ be an elliptic curve given by the equation $y^{2}= x^{3} +Ax + B$, where $A$ and $B$ are integers satisfying $\Delta = 4A^{3} + 27 B^{2}\neq 0$. Suppose that $E$ does not have complex multiplication. Then there exist constants $C(E)>0$ and $x_{0}(E)>0$ depending only on the elliptic curve $E$ such that for any $x\geq x_{0}(E)$, $1 \leq n \leq (\log x)^{1/14}$, and any $\alpha\in \mathbb{Z}$ we have
\[
\#\{3\leq q\leq x: \#E(\mathbb{F}_{q})\equiv \alpha\ \text{\textup{(mod} $n$\textup{)}}\}\leq
C(E) \biggl(\frac{x}{\varphi(n)\log x}+ x\, \textup{exp} (-b n^{-2} \sqrt{\log x})\biggr).
\]Here $b>0$ is an absolute constant.
\end{lemma}
\begin{proof}
This Lemma follows from part \textup{(i)} of Theorem 2.3 in \cite{David.Wu}.
\end{proof}

\begin{lemma}\label{LEMMA.ELLIPTIC}
Let $E$ be an elliptic curve given by the equation $y^{2}= x^{3} +Ax + B$, where $A$ and $B$ are integers satisfying $\Delta = 4A^{3} + 27 B^{2}\neq 0$. Suppose that $E$ does not have complex multiplication. Let $f(n)=\gamma_{d} n^{d}+\ldots + \gamma_{0}$ be a polynomial with integer coefficients such that $\gamma_{d}>0$ and $f: \mathbb{N}\to \mathbb{N}$. Set $\mathcal{B}=\{f(\#E(\mathbb{F}_{p})): p\geq 3\}$. Then for sufficiently large $x$ (in terms of $E$ and $f$) we have
 \begin{equation}\label{ELLIPT.ORD}
 \mathcal{B}(x)\asymp \frac{x^{1/d}}{\log x},\qquad 1\leq \rho_{\mathcal{B}}(x) \ll x^{1/(2d)},
 \end{equation}and
 \[
\sum_{\substack{j\in \mathbb{P}:\,j\geq 3\\ f(\#E(\mathbb{F}_{j})) \leq x}} \sum_{p\leq (\log x)^{1/15}} \frac{\lambda_{\mathcal{B}} (x;j, p)\log p}{p}
\ll  \mathcal{B}(x)^{2},
\] where the constants implied by the symbols $\asymp$ and $\ll$ depend only on the elliptic curve $E$ and the polynomial $f$.
\end{lemma}
\begin{proof}

By Hasse's theorem (see, for example, \cite[Theorem 477]{Hardy.Wright}), we have $|\#E(\mathbb{F}_{p})- (p+1)|< 2\sqrt{p}$, and hence
\begin{equation}\label{HASSE}
(\sqrt{p}-1)^{2}< \#E(\mathbb{F}_{p})< (\sqrt{p}+1)^{2}.
\end{equation}

Suppose that $q\geq 5$. Let $p$ be a prime with $p> q + 4\sqrt{q}+ 4$. Then $(\sqrt{p}-1)^{2}> (\sqrt{q}+1)^{2}$, and by \eqref{HASSE} we have $\#E(\mathbb{F}_{p})> \#E(\mathbb{F}_{q})$. Let $p$ be a prime with $p< q - 4\sqrt{q}+ 4$. Then $(\sqrt{p}+1)^{2}< (\sqrt{q}-1)^{2}$, and it follows from \eqref{HASSE} that $\#E(\mathbb{F}_{p})< \#E(\mathbb{F}_{q})$. Hence,
\begin{align*}
\#\{p \geq 3: \#E(\mathbb{F}_{p})=\#E(\mathbb{F}_{q})\}&\leq \#\{p: q - 4\sqrt{q}+ 4 \leq p \leq q + 4\sqrt{q}+ 4 \}\\
&\leq \#\{n\in \mathbb{N}: q - 4\sqrt{q}+ 4 \leq n \leq q + 4\sqrt{q}+ 4 \}\\
&\leq 8\sqrt{q}+1 < 9\sqrt{q}.
\end{align*}Applying \eqref{HASSE} again, we see that $\#E(\mathbb{F}_{p})>\#E(\mathbb{F}_{3})$ for any $p>13$. Therefore
\[
\#\{p \geq 3: \#E(\mathbb{F}_{p})=\#E(\mathbb{F}_{3})\}\leq \#\{p\leq 13\}=6.
\]We obtain
\begin{equation}\label{BASIC.ORDER}
\#\{p \geq 3: \#E(\mathbb{F}_{p})=\#E(\mathbb{F}_{q})\} \leq 9\sqrt{q}
\end{equation}for any $q\geq 3$.

 Let $v\in \mathbb{N}$ and $\alpha_{1},\ldots,\alpha_{l}$ be all positive integer roots of the equation $f(x)=v$. We note that $l\leq d$. We have
\begin{equation}\label{ORDER}
\textup{ord}_{\mathcal{B}} (v)=\#\{p\geq 3: f(\#E(\mathbb{F}_{p}))=v\} \leq \sum_{i=1}^{l}
\#\{p\geq 3: \#E(\mathbb{F}_{p})= \alpha_{i}\}.
\end{equation}

Fix $1\leq i \leq l$. Applying \eqref{POLYNOMIAL.c.1.c.2}, we have
\[
v= f(\alpha_{i})\geq c_{1} \alpha_{i}^{d},\qquad\text{and hence}\qquad \alpha_{i} \leq \Big(\frac{v}{c_{1}}\Big)^{1/d}.
\]If $\alpha_{i}\neq \#E(\mathbb{F}_{q})$ for any $q\geq 3$, then
\[
\#\{p\geq 3: \#E(\mathbb{F}_{p})= \alpha_{i}\} = 0.
\] Suppose that $\alpha_{i} = \#E(\mathbb{F}_{q})$ for some $q\geq 3$. Then we have
\[
\Big(\frac{v}{c_{1}}\Big)^{1/d} \geq \alpha_{i} = \#E(\mathbb{F}_{q}) > (\sqrt{q}-1)^{2}> \frac{q}{9}.
\]We get
\begin{equation}\label{BASIC.q}
q \ll v^{1/d}.
\end{equation} From \eqref{BASIC.ORDER} and \eqref{BASIC.q} we obtain
\[
\#\{p\geq 3: \#E(\mathbb{F}_{p})= \alpha_{i}\}\ll v^{1/(2d)}.
\]Since $l\leq d$, from \eqref{ORDER} we find that
\[
\textup{ord}_{\mathcal{B}}(v)\ll v^{1/(2d)}
\]for any positive integer $v$. Therefore for sufficiently large $x$ in terms of $f$ and $E$ we have
\[
1 \leq \rho_{\mathcal{B}}(x) = \max_{1\leq v \leq x}\textup{ord}_{\mathcal{B}}(v) \ll x^{1/(2d)}.
\]

Applying \eqref{POLYNOMIAL.c.1.c.2} and \eqref{HASSE}, we obtain
\begin{align}
\mathcal{B}(x)&=\#\{p\geq 3: f(\#E(\mathbb{F}_{p}))\leq x\}\leq \#\{p\geq 3: c_{1}(\#E(\mathbb{F}_{p}))^{d}\leq x\}\notag\\
&= \#\Big\{p\geq 3: \#E(\mathbb{F}_{p})\leq \Big(\frac{x}{c_{1}}\Big)^{1/d}\Big\}\leq
\#\Big\{p\geq 3: (\sqrt{p}-1)^{2}\leq \Big(\frac{x}{c_{1}}\Big)^{1/d}\Big\}\notag\\
&\leq \#\Big\{p\geq 3: p\leq 9\Big(\frac{x}{c_{1}}\Big)^{1/d}\Big\}\ll \frac{x^{1/d}}{\log x}.\label{B.LOW}
\end{align}
Similarly,
\begin{align*}
\mathcal{B}(x)&\geq \#\{p\geq 3: c_{2}(\#E(\mathbb{F}_{p}))^{d}\leq x\}= \#\Big\{p\geq 3: \#E(\mathbb{F}_{p})\leq \Big(\frac{x}{c_{2}}\Big)^{1/d}\Big\}\\
&\geq \#\Big\{p\geq 3: (\sqrt{p}+1)^{2}\leq \Big(\frac{x}{c_{2}}\Big)^{1/d}\Big\}\\
&\geq \#\Big\{p\geq 3: p\leq \frac{1}{4}\Big(\frac{x}{c_{2}}\Big)^{1/d}\Big\}\gg \frac{x^{1/d}}{\log x}.
\end{align*}Thus, the statement \eqref{ELLIPT.ORD} is proved.

Fix $j\in \mathbb{P}$ such that $j\geq 3$ and $f(\#E (\mathbb{F}_{j}))\leq x$ and $p\leq (\log x)^{1/15}$. Suppose that $p>  \gamma_{d}$, where $\gamma_{d}$ is the leading coefficient of the polynomial $f$. Let $\alpha_{1},\ldots, \alpha_{l}$ be all solutions of the congruence
\[
f(x)\equiv f(\#E(\mathbb{F}_{j}))\text{ (mod $p$)}.
\]By Lagrange's theorem, $l\leq d$. Set
\[
N= \Big[9\Big(\frac{x}{c_{1}}\Big)^{1/d}\Big].
\]

 We see from \eqref{B.LOW} that
\begin{align*}
\lambda_{\mathcal{B}}(x;j, p)&=\#\{q\geq 3:  f(\#E(\mathbb{F}_{q}))\leq x\text{ and } f(\#E(\mathbb{F}_{q})) \equiv f(\#E(\mathbb{F}_{j}))\text{ (mod $p$)}\}\\
&=\sum_{i=1}^{l}\#\{q\geq 3:  f(\#E(\mathbb{F}_{q}))\leq x\text{ and } \#E(\mathbb{F}_{q}) \equiv \alpha_{i}\text{ (mod $p$)}\}\\
&\leq\sum_{i=1}^{l}\#\Big\{ 3\leq q\leq N:
  \#E(\mathbb{F}_{q}) \equiv \alpha_{i}\text{ (mod $p$)}\Big\}.
\end{align*} For sufficiently large $x$, we have
\[
N\geq x_{0}(E),
\]where $x_{0} (E)$ is the constant from Lemma \ref{LEMMA.DAVID.WU}, and
\[
p\leq (\log x)^{1/15}\leq (\log N)^{1/14}.
\] Using Lemma \ref{LEMMA.DAVID.WU}, we obtain
\[
\lambda_{\mathcal{B}}(x;j, p)\ll \frac{x^{1/d}}{p\log x}+ x^{1/d}\, \textup{exp} \big(-b(f) p^{-2} \sqrt{\log x}\big),
\] where $b(f)>0$ is a constant depending only on $f$, and the constant implied by the symbol $\ll$ depends only on $E$ and $f$.

 We have
\begin{align*}
\sum_{\gamma_{d}< p\leq (\log x)^{1/15}}\frac{\lambda_{\mathcal{B}}(x;j, p)\log p}{p}\ll&
\frac{x^{1/d}}{\log x} \sum_{\gamma_{d}< p\leq (\log x)^{1/15}}\frac{\log p}{p^{2}}\\
&+x^{1/d}\sum_{\gamma_{d}< p\leq (\log x)^{1/15}}\textup{exp} \big(-b(f) p^{-2} \sqrt{\log x}\big).
\end{align*} For any $\gamma_{d}< p\leq (\log x)^{1/15}$, we have
\[
\textup{exp} \big(-b(f) p^{-2} \sqrt{\log x}\big)\leq \textup{exp} \big(-b(f) (\log x)^{11/30}\big).
\]We obtain that the last sum is
\[
\leq x^{1/d}(\log x)^{1/15}\textup{exp} \big(-b(f) (\log x)^{11/30}\big) \leq \frac{x^{1/d}}{\log x}.
\]Since
\[
\sum_{\gamma_{d}< p\leq (\log x)^{1/15}}\frac{\log p}{p^{2}}<
\sum_{p}\frac{\log p}{p^{2}}\ll 1,
\]we get
\begin{equation}\label{STEP.1}
\sum_{\gamma_{d}< p\leq (\log x)^{1/15}}\frac{\lambda_{\mathcal{B}}(x;j, p)\log p}{p}
\ll \frac{x^{1/d}}{\log x} \ll \mathcal{B}(x).
\end{equation} Since
\[
\lambda_{\mathcal{B}}(x;j, p) \leq \mathcal{B}(x),
\]we obtain
\begin{equation}\label{STEP.2}
\sum_{p \leq \gamma_{d}}\frac{\lambda_{\mathcal{B}}(x;j, p)\log p}{p}\leq
\mathcal{B}(x) \sum_{ p \leq \gamma_{d}} \frac{\log p}{p}\ll \mathcal{B}(x).
\end{equation}Gathering \eqref{STEP.1} and \eqref{STEP.2} together, we obtain
\[
\sum_{p \leq (\log x)^{1/15}}\frac{\lambda_{\mathcal{B}}(x;j, p)\log p}{p} \ll \mathcal{B}(x),
\]where the constant implied by the symbol $\ll$ depends only on $E$ and $f$. Therefore
\[
\sum_{\substack{j\in \mathbb{P}:\,j\geq 3\\ f(\#E(\mathbb{F}_{j})) \leq x}} \sum_{p\leq (\log x)^{1/15}} \frac{\lambda_{\mathcal{B}} (x;j, p)\log p}{p}
\ll  \mathcal{B}(x) \sum_{\substack{j\in \mathbb{P}:\,j\geq 3\\ f(\#E(\mathbb{F}_{j})) \leq x}}1 =
\mathcal{B}(x)^{2}.
\]Lemma \ref{LEMMA.ELLIPTIC} is proved.
\end{proof}

We see from Lemma \ref{LEMMA.ELLIPTIC} that the sequence $\mathcal{B}=\{f(\#E(\mathbb{F}_{p})): p\geq 3\}$ satisfies \eqref{T1:Basic.2} and \eqref{T1:Basic.3} with $\alpha = 1/15$. Application of Theorem \ref{T.GENERAL} completes the proof of Theorem \ref{T.ELLIPTIC.POLYNOMIAL}.

\begin{proof}[Proof of Corollary \ref{COROLLARY.ELLIPTIC.POLYNOMIAL}.] We set
 \[
 \mathcal{A}_{1}=\mathbb{P},\qquad \mathcal{A}_{2}= \mathcal{S}',\qquad \mathcal{B}=\{f(\#E(\mathbb{F}_{p})): p\geq 3\}.
 \]By Examples 1 and 2, the sequences $\mathcal{A}_{1}$ and $\mathcal{A}_{2}$ satisfy the hypotheses of Theorem \ref{T.ELLIPTIC.POLYNOMIAL}.  Applying this theorem, we obtain \eqref{REFER.ELL.1} and
 \begin{equation}\label{BASIC.T.f.p.III}
 \#\Big\{n \leq x: r'_{2}(n) \geq c_{1}\frac{x^{1/d}}{(\log x)^{3/2}}\Big\} \gg x,
 \end{equation}where
 \[
 r'_{2}(n)=\#\{ (s,p)\in \mathcal{S}'\times \mathbb{P}: p\geq 3\text{ and } s+ f(\#E(\mathbb{F}_{p}))=n\}.
 \]Since $r'_{2}(n)\leq r_{2}(n)$, from \eqref{BASIC.T.f.p.III} we obtain \eqref{REFER.ELL.2}. Corollary \ref{COROLLARY.ELLIPTIC.POLYNOMIAL} is proved.
 \end{proof}

 \section{Proof of Theorem \ref{T.EULER.ELLIPTIC}}

 We assume that $x\geq x_{0}$, where $x_{0}= x_{0}(E, f)$ is a large constant depending only on $E$ and $f$. By \eqref{POLYNOMIAL.c.1.c.2} and \eqref{HASSE}, for any prime $3 \leq q\leq x$ we have
\begin{equation}\label{M.ESTIMATE}
f(\#E(\mathbb{F}_{q}))\leq c_{2} (\#E(\mathbb{F}_{q}))^{d}\leq c_{2}(\sqrt{x}+1)^{2d}\leq c x^{d},
\end{equation}where $c >0$ is a constant depending only on $f$.

 We put $M=cx^{d}$. We have
 \[
 (\log M)^{1/15} = (d\log x+ \log c)^{1/15}< (\log x)^{1/14}.
 \] Applying Lemma \ref{Lemma.Euler} with $\alpha =1/15$, we obtain
\begin{equation}\label{EULER.1}
\sum_{3 \leq q \leq x} \bigg(\frac{f(\#E(\mathbb{F}_{q}))}{\varphi(f(\#E(\mathbb{F}_{q})))}\bigg)^{s}
\leq C_{0}^{s}\bigg(\pi (x)+ \sum_{p\leq (\log x)^{1/14}}\frac{\omega (p)(\log p)^{s}}{p}\bigg),
\end{equation}where $C_{0}>0$ is an absolute constant and
\[
\omega (p) = \#\big\{3\leq q\leq x: f(\#E(\mathbb{F}_{q}))\equiv 0\textup{ (mod $p$)}\big\}.
\]

Fix $\gamma_{d}< p \leq (\log x)^{1/14}$, where $\gamma_{d}$ is the leading coefficient of the polynomial $f$. Let $\alpha_{1},\ldots, \alpha_{l}$ be all solutions of the congruence $f(x)\equiv 0$ (mod $p$). By Lagrange's theorem, $l\leq d$. Using Lemma \ref{LEMMA.DAVID.WU}, we obtain
\[
\omega (p) =\sum_{i=1}^{l} \#\big\{3\leq q\leq x: \#E(\mathbb{F}_{q})\equiv \alpha_{i}\textup{ (mod $p$)}\big\}\ll \frac{x}{p\log x}+ x\, \textup{exp} (-b p^{-2} \sqrt{\log x}),
\]where the constant implied by the $\ll$-symbol depends only on $E$ and $f$. Therefore
\begin{align}
\sum_{ \gamma_{d}< p\leq (\log x)^{1/14}}\frac{\omega (p)(\log p)^{s}}{p}
\ll\  &\pi(x)\sum_{\gamma_{d}< p\leq (\log x)^{1/14}} \frac{(\log p)^{s}}{p^{2}}\notag\\
&+ x \sum_{\gamma_{d}< p\leq (\log x)^{1/14}}\textup{exp} (-b p^{-2} \sqrt{\log x})\label{EULER.2}
\frac{(\log p)^{s}}{p}.
\end{align}

Applying summation by parts, for any positive integer $N> (\log x)^{1/14}$ we have
\begin{align}
&\sum_{\gamma_{d}< p\leq (\log x)^{1/14}}\frac{(\log p)^{s}}{p^{2}}\leq \sum_{1\leq n \leq N} \frac{(\log n)^{s}}{n^{2}}=\int_{1}^{N}\frac{(\log t)^{s}}{t^{2}}\,dt +
\int_{1}^{N}\{t\} \frac{s(\log t)^{s-1}}{t^{3}}\, dt\notag\\
 &- \int_{1}^{N}\{t\}\frac{2(\log t)^{s}}{t^{3}}\, dt
 \leq \int_{1}^{\infty}\frac{(\log t)^{s}}{t^{2}}\,dt + s
\int_{1}^{\infty} \frac{(\log t)^{s-1}}{t^{2}}\, dt\notag\\
&=\int_{0}^{\infty}e^{-u} u^{s}\,du + s\int_{0}^{\infty}e^{-u} u^{s-1}\,du = \Gamma(s+1)+ s \Gamma(s)\leq 2\,\textup{exp}(s\log s)\label{EULER.s!}
\end{align} (we used the trivial bound $s!\leq s^{s}$).

Since
\[
\omega(p) \leq \pi (x),
\]we have
\begin{equation}\label{EULER.4}
\sum_{p\leq\gamma_{d}}\frac{\omega(p)(\log p)^{s}}{p} \leq C_{1}^{s} \pi (x),
\end{equation}where $C_{1}>0$ is a constant depending only on $E$ and $f$.

 For any $\gamma_{d}< p\leq (\log x)^{1/14}$ we have
\[
\textup{exp} (-b p^{-2} \sqrt{\log x})\leq \textup{exp} \big(-b (\log x)^{1/2 - 1/7}\big)
\leq \textup{exp} \big(-(\log x)^{1/3}\big),
\]if $x_{0}$ is sufficiently large. Set
\[
A(t) := \sum_{p \leq t} (\log p)^{s}.
\]Then for any $t\geq 2$ we have
\[
A(t) \leq (\log t)^{s} \pi(t) \leq a t (\log t)^{s-1},
\]where $a>0$ is an absolute constant. Applying partial summation, for any positive integer $N\geq 2$ we obtain
\[
\sum_{p \leq N} \frac{(\log p)^{s}}{p}=\frac{A(N)}{N} + \int_{2}^{N}\frac{A(t)}{t^{2}}\, dt
\leq a (\log N)^{s-1} + a \int_{2}^{N}\frac{(\log t)^{s-1}}{t}\, dt < a_{1} (\log N)^{s},
\]where $a_{1}>0$ is an absolute constant. We obtain
\begin{equation}
x \sum_{ \gamma_{d}< p\leq (\log x)^{1/14}}\textup{exp} (-b p^{-2} \sqrt{\log x})
\frac{(\log p)^{s}}{p}\leq a_{1}\Big(\frac{1}{14}\Big)^{s} \frac{x (\log\log x)^{s}}{\textup{exp} ((\log x)^{1/3})}.\label{EULER.3}
\end{equation}

 Gathering \eqref{EULER.1}--\eqref{EULER.3} together, we obtain
\[
\sum_{3 \leq q \leq x} \bigg(\frac{f(\#E(\mathbb{F}_{q}))}{\varphi(f(\#E(\mathbb{F}_{q})))}\bigg)^{s}
\leq\textup{exp}(s\log s)C_{2}^{s}\frac{x}{\log x} + C_{2}^{s} \frac{x (\log\log x)^{s}}{\textup{exp} ((\log x)^{1/3})}= I + II
\]for any $x\geq x_{0}, $where $C_{2}>0$ is a constant depending only on $E$ and $f$.

Let us show that $II \leq I$. It suffices to show that
\begin{equation}\label{BASIC.s.ART}
s \log\log\log x + \log\log x - (\log x)^{1/3} \leq s\log s.
\end{equation} We first note the following trivial inequality. If $c>0$, then
\begin{equation}\label{BASIC.c.PREP.ART}
y\log y - c y \geq - \textup{exp}(c-1)
\end{equation}for all $y>0$. Indeed, consider the function $h(y) = y\log y - cy$. Then we have $h' (y) = \log y - (c-1)$. We obtain
\[
h(y) \geq h(\textup{exp}(c-1))=- \textup{exp}(c-1)
\]and \eqref{BASIC.c.PREP.ART} is proved. Applying \eqref{BASIC.c.PREP.ART} with $c=\log\log\log x$, we obtain
\begin{align*}
s\log s - s \log\log\log x + &(\log x)^{1/3} - \log\log x\\
&\geq
-\textup{exp}(\log\log\log x - 1)+ (\log x)^{1/3} - \log\log x>0,
\end{align*}if $x_{0}$ is sufficiently large. Thus, \eqref{BASIC.s.ART} is proved, and hence $II \leq I$.

 We obtain
\[
\sum_{3 \leq q \leq x} \bigg(\frac{f(\#E(\mathbb{F}_{q}))}{\varphi(f(\#E(\mathbb{F}_{q})))}\bigg)^{s}
\leq\textup{exp}(s\log s + Cs)\pi(x)
\]for any $x \geq x_{0}$, where $C>0$ is a constant depending only on $E$ and $f$.

If $3\leq x \leq x_{0}$, then, by \eqref{M.ESTIMATE}, for any $3 \leq q \leq x$ we have
\[
\frac{f(\#E(\mathbb{F}_{q}))}{\varphi (f(\#E(\mathbb{F}_{q})))}\leq f(\#E(\mathbb{F}_{q}))\leq c x_{0}^{d}=:D,
\]where $D>0$ is a constant depending only on $E$ and $f$. Hence,
\[
\sum_{3 \leq q \leq x} \bigg(\frac{f(\#E(\mathbb{F}_{q}))}{\varphi(f(\#E(\mathbb{F}_{q})))}\bigg)^{s}\leq
D^{s} \pi (x).
\] Theorem \ref{T.EULER.ELLIPTIC} is proved.

\section{Acknowledgements}
The author is grateful to the anonymous reviewers for their helpful comments.
This work is an output of a research project (HSE-BR-2025-024) implemented as part of the Basic Research Program at HSE University.

\end{document}